\documentclass[11pt]{amsart}

\usepackage{amsfonts,amssymb,amsthm}
\usepackage[mathscr]{eucal}
\usepackage{palatino, mathpazo}
\usepackage{mathrsfs}
\usepackage[all]{xy}
\usepackage[
]{hyperref}

\newtheorem{thm}{Theorem}[subsection]
\newtheorem{lemma}[thm]{Lemma}
\newtheorem{prop}[thm]{Proposition}
\newtheorem{cor}[thm]{Corollary}

\newtheorem{defn}[thm]{Definition}

\newtheorem{question}[thm]{Question}

\theoremstyle{remark}
\newtheorem{remark}[thm]{Remark}

\numberwithin{equation}{subsection}

\newcommand{\A}{\mathbb{A}}
\newcommand{\cc}{\mathbb{C}}

\newcommand{\zz}{\mathbb{Z}}
\newcommand{\bL}{\mathbb{L}}

\newcommand{\bp}{\mathbb{P}}

\newcommand{\K}{\mathcal{K}}
\newcommand{\M}{\mathcal{M}}

\newcommand{\T}{\mathscr{F}}
\newcommand{\oo}{\mathscr{O}}

\newcommand{\phib}{\bar{\psi}}
\newcommand{\taub}{\bar{\tau}}

\newcommand{\mbar}{\overline{\M}}
\newcommand{\on}{\operatorname}
\newcommand{\ft}{\rho}

\newcommand{\llangle}{ \Big[  } 
\newcommand{\rrangle}{ \Big] } 

\DeclareMathOperator{\bl}{Bl}
\DeclareMathOperator{\st}{st}
\DeclareMathOperator{\pr}{pr}
\DeclareMathOperator{\Fl}{Fl}
\DeclareMathOperator{\vir}{vir}
\DeclareMathOperator{\Tor}{\mathcal{T}\!\emph{or}}
\DeclareMathOperator{\Spec}{Spec}
\DeclareMathOperator{\aut}{Aut}
\DeclareMathOperator{\ev}{ev}

\author[Y.-P.~Lee]{Yuan-Pin~Lee}
\address{Y.-P.~Lee: Department of Mathematics, University of Utah,
Salt Lake City, Utah 84112-0090, U.S.A.}
\email{yplee@math.utah.edu}

\author[H.-W.~Lin]{Hui-Wen~Lin}
\address{H.-W.~Lin: Department of Mathematics and Taida
Institute of Mathematical Sciences (TIMS),
National Taiwan University, Taipei 10617, Taiwan}
\email{linhw@math.ntu.edu.tw}

\author[F.~Qu]{Feng~Qu}
\address{F.\ Qu: Department of Mathematics, University of Utah,
Salt Lake City, Utah 84112-0090, U.S.A.}
\email{qu@math.utah.edu}

\author[C.-L.~Wang]{Chin-Lung~Wang}
\address{C.-L.\ Wang: Department of Mathematics,
Center for Advanced Studies in Theoretical Sciences (CASTS), and Taida
Institute of Mathematical Sciences (TIMS),
National Taiwan University, Taipei 10617, Taiwan}
\email{dragon@math.ntu.edu.tw}


\title[Invariance of Quantum Rings III]{Invariance of quantum rings\\ 
under ordinary flops: III}

\begin{document}

\maketitle

\begin{abstract}
The paper is a sequel to \cite{LLW1, LLW2}, as part of our project to study a case of Crepant Transformation Conjecture: $K$-equivalence Conjecture for ordinary flops. In this paper we prove the invariance of quantum rings for general ordinary flops, whose local models are certain \emph{non-split} toric bundles over \emph{arbitrary} smooth base. 
An essential ingredient in the proof is a \emph{quantum splitting principle} which reduces a statement in Gromov--Witten theory on non-split bundles to
the case of split bundles.
\end{abstract}

\small
\tableofcontents
\normalsize

\section{Introduction} \label{s:0}
The paper is Part III of our ongoing efforts to establish the \emph{$K$-equivalence conjecture} \cite{Wang2, Wang3} for general ordinary flops, a special but pivotal case of the \emph{Crepant Transformation Conjecture}. (See Section \ref{s:1} for definition of ordinary flops.) In Part I and Part II \cite{LLW1, LLW2}, we solve the genus zero case for which the local models are \emph{split} projective bundles over smooth bases. In this paper, we extend the result to general ordinary flops whose local models are \emph{non-split} projective bundles over \emph{arbitrary} smooth bases. As far as we know, this is the first general result for non-split bundles (which can not be transformed via deformations etc.\ into split bundles) in Gromov--Witten theory. The previous results are mostly based on localization of fiberwise $\cc^*$ action for which bundles must be direct sums of line bundles. We believe that the techniques developed here can be applied to other problems in Gromov--Witten theory. For example, they will form the backbone of the \emph{quantum splitting principle} \cite{FQ}. Meanwhile, the full procedure on \emph{analytic continuations of quantum cohomolgy rings} initiated in \cite{LLW} and completed here, together with the \emph{inductive structure on stratified flops} established in \cite{FW}, will form the foundation to attack the $K$-equivalence conjecture.    

\subsection{Main results}
Recall that given an ordinary flop 
$$
f: X \dashrightarrow X'
$$ 
there is a canonically induced isomorphism of Chow motives by the graph closure $\bar\Gamma_f$. See \cite{LLW, LLW1} for definitions and results. 
In particular, it induces an isomorphism of Chow groups and Cohomology groups 
$$
\T = [\bar\Gamma_f]_*: H(X) \to H(X'),
$$ 
with the Poincar\'e pairing preserved. To extend the correspondence to the context of Gromov--Witten theory, the quantum variables still need to be identified in order for the comparison to work. To that end, we set
\[
 \T (q^{\beta})= q^{\T(\beta)}.
\]
Our main result is the following theorem.

\begin{thm} \label{t:main}
For a general ordinary flop $X \dashrightarrow X'$, $\T$ induces an isomorphism of big quantum rings of $X$ and $X'$ after an analytic continuation over the Novikov variables corresponding to the extremal rays. Furthermore, the same results hold for relative primary invariants and (relative) ancestors.
\end{thm}

Some explanation is in order. Firstly, the cohomological ring structures are not preserved under $\T$ \cite{LLW}. 
Since the quantum ring is a deformation of the classical ring, it might seem impossible to have isomorphic quantum rings without isomorphic classical ones. Secondly, for an effective curve class $\beta$ in $X$, $\T (\beta)$ is in general not an effective curve in $X'$, and in that case $q^{\T(\beta)}$ is not in the Novikov ring of $X'$. These two problems are solved simultaneously with \emph{analytic continuation}. More precisely, the comparison of Gromov--Witten theory is only valid between generating functions summing over contributions from the extremal rays. See Section~\ref{s:1.2} for definition of the term \emph{generating functions}. Theorem~\ref{t:main} implies that these generating functions on $X$ are analytic in $q^{\ell}$, the Novikov variable corresponding to the extremal ray $\ell$. Similarly, the corresponding generating functions on $X'$ are analytic in $\ell'$. In this sense $\T$ identifies the corresponding generating functions on $X$ and on $X'$ \emph{as analytic functions in $q^{\ell}$} (while remaining formal in other Novikov variables). 

We remark that analytic continuations can also be formulated on the complexified K\"ahler moduli as  done in \cite{LLW, ILLW, LLW1}. We note also that Theorem~\ref{t:main} does not hold for descendants. Examples are given in \cite{LLW}.

\subsection{Outline of the strategy} \label{strategy}
The strategy of the proof involves steps of reduction.
In Part I \cite{LLW1}, a degeneration argument together with various reconstruction results reduce the proof to 
the corresponding statements on the \emph{local models}.
We note that a local model of an ordinary flop is constructed from a triple
$(S,F,F')$ consisting of two vector bundles $F$ and $F'$ of equal rank
over a smooth projective variety $S$. Indeed, the $f$ exceptional loci $Z \subset X$ and $Z' \subset X'$ are projective bundles 
$$\phib: Z = \Bbb P_S(F) \to S, \quad \phib': Z' = \Bbb P_S(F') \to S,$$ 
and the local models are toric (double projective) bundles over $S$:
\[
 \begin{split}
 &X = \Bbb P_{Z}(\mathscr{O}_Z(-1) \otimes \phib^* F' \oplus \mathscr{O}), \\
 &X' = \Bbb P_{Z'}(\mathscr{O}_{Z'}(-1) \otimes \phib'^*F \oplus \mathscr{O}).
 \end{split}
\]
The flop $f: X \dasharrow X'$ is the blowup of $X$ along $Z$ followed by contracting the exceptional divisor along the other ruling. The local model of ordinary flops can be viewed as a functor over the triple $(S, F, F')$. (See Section~\ref{s:1.1}.) 

In the next step, we modify the triple by blowing up the base $S$ with the aim of simplifying the structure of the bundles $F$ and $F'$.
Starting with 
$$
(S_0, F_0, F'_0) = (S,F,F'),
$$ 
we construct a sequence of triples $(S_i, F_i, F'_i)_{i \ge 0}$, such that $S_{i+1}$ is obtained by blowing up $S_i$ along some smooth subvariety $Z_i$, and $F_{i+1}$ and 
$F'_{i+1}$ are the pullback of $F_i$ and $F'_i$ from $S_i$ to $S_{i+1}$ respectively. 
We will show
\begin{itemize}
\item[(i)] {\it $\T$-invariance for $(S_i, F_i, F'_i)$ can be reduced to the $\T$-invariance for the triple in the next stage $(S_{i+1}, F_{i+1}, F'_{i+1})$.} 

\item[(ii)] {\it After a finite number of blowups, we obtain a triple $(S_n, F_n, F'_n)$ 
such that $F_n$ and $F'_n$ can be deformed to a direct sum of line bundles.}
\end{itemize}

That is, we end up with an ordinary flop of \emph{splitting type}, for which the $\T$-invariance is proved in Part II via a \emph{quantum Leary--Hirsch theorem for split toric bundles} \cite{LLW2}. Theorem \ref{t:main} then follows from (i). 

Recall that the \emph{Quantum Leray--Hirsch} says that for $X \to S$ a split (iterated) projective bundle, the Dubrovin connection on $X$ can be constructed from a (carefully chosen) lifting of 
the Dubrovin connection on $S$ and the Picard--Fuchs system associated to the fiber. When $f: X \dasharrow X'$ is an ordinary flop, the naturality of the construction in \cite{LLW2} allows us to perform analytic continuations along the fiberwise Novikov variables by way of the Picard--Fuchs systems on $X$ and $X'$.
They turn out to coincide after analytic continuations. 

The \emph{splitting principle} in (ii) via blow-ups is in fact a simple consequence of Hironaka's theorem on resolution of indeterminacies (cf.~Section~\ref{splitting}). Thus the major efforts made in this paper is indeed to prove (i). To relate the $\T$-invariance for  $(S_i, F_i, F'_i)$ to that for $(S_{i+1}, F_{i+1}, F'_{i+1})$, we consider the deformation to the normal cone for $T_i \hookrightarrow S_i$. This is the family 
\[ 
\bl_{T_i \times \{0\}}(S_i \times \A^1) \to \A^1 
\]
with smooth fiber $S_i$ over $\A^1-\{0\}$ and singular fiber 
$$
S_{i+1} \cup _{E_{i}} P_i,
$$ 
where $E_i$ is the exceptional divisor in $S_{i+1}=\bl_{T_i}S_i$, and $P_i$ is the exceptional divisor in $\bl_{T_i \times \{0\}}(S_i \times \A^1)$. By construction, $P_i$ and $E_i$ are themselves projective bundles over $T_i$:
$$
P_i = \Bbb P_{T_i}(N_{T_i/S_i} \oplus \mathscr{O}), \qquad E_i = \Bbb P_{T_i}(N_{T_i/S_i}). 
$$
The bundles  $\bar{F}_{i+1}$ and $\bar{F}'_{i+1}$ on $P_i$ and $E_i$ are pulled backs from the induced bundles on $T_i$, which are compatible with the restrictions of $F_{i+1}$ and $F'_{i+1}$ on $S_{i+1}$. In order to relate the $\T$-invariance for $(S_i, F_i, F'_i)$ with 
that for $(S_{i+1}, F_{i+1}, F'_{i+1})$ and $(P_i, \bar{F}_{i+1}, \bar{F}'_{i+1})$ (as well as for $(E_i, \bar{F}_{i+1}, \bar{F}'_{i+1})$), the degeneration formula \cite{LiRu, Li} is used,
and the corresponding statements must be generalized to relative Gromov--Witten invariants of smooth pairs.

To avoid cumbersome notations, \emph{we will omit the bundles in the notation} when there is no danger of confusion. 
By the degeneration formula, $\T$-invariance for $S_i$ (absolute invariants) is implied by those for $(S_{i+1}, E_i)$ and $(P_i, E_i)$ (relative invariants). We then show that 
\begin{itemize}
\item[(iii)] {\it$\T$-invariance for relative invariants on $(S_{i+1}, E_i)$ follows from the $\T$-invariance for absolute invariants on $S_{i+1}$ and $E_i$. Similarly $\T$-invariance for $(P_i, E_i)$ follows from that for $P_i$ and $E_i$.}
\end{itemize}

This is the most intricate part of the reduction. Before we discuss it, we explain why (iii) allows us to run an inductive proof on the dimension of the base $S$. Clearly $\dim E_i = \dim S_i - 1$ and $\T$-invariance for $E_i$ follows by induction. However, $\dim P_i = \dim S_i$ does not drop. A key observation is that $P_i = \Bbb P_{T_i}(N_{T_i/S_i} \oplus \mathscr{O})$ is constructed from $T_i$ and $\dim T_i < \dim S_i$. If $N_{T_i/S_i}$ splits, then a variant of the quantum Leray--Hirsch still applies. Namely analytic continuations along the base $T_i$ ($\T$-invariance) implies the $\T$-invariance for $P_i$. In fact, the fiberwise Novikov variables in $P_i \to T_i$ need no analytic continuations at all. However, the bundle $N_{T_i/S_i}$ may not be deformable to split bundles. This suggests that we should have taken into account the splitting procedure of those relevant normal bundles during the induction process. It is thus plausible to expect that some kind of \emph{refined induction procedure} may lead to $\T$-invariance for $P_i$. Such a delicate induction is indeed possible, as in Section~\ref{TheProof}. 
Hence $\T$-invariance for $S_{i+1}$ implies that for $S_i$ under the validity of (iii). Note that in the precess we have to go back and forth between absolute and relative invariants 
\emph{without involving descendants} in order to keep the $\T$-invariance in all inductive steps.

It remains to prove (iii). 
Similar ideas were already used by Maulik and Pandharipande in \cite{MP}, where it was shown that the relative descendants of a \emph{smooth pair} $(S, D)$ ($D$ a smooth divisor in $S$) 
are determined by the absolute descendants of $S$ and $D$. This is accomplished by considering the \emph{trivial} deformations to the normal cone 
\[ \bl_{D \times \{0\}}(S \times \A^1) \to \A^1 \]
with smooth fiber $S$ over $\A^1 \setminus\{0\}$ and singular fiber $S \cup _{D} P$, where 
$$P = \Bbb P_D(N_{D/S} \oplus \mathscr{O})$$ 
is a projectivized split bundle of rank two. 
In our case the degeneration formula (which we recall in Section~\ref{deg-formula}) decompose absolute invariants into convolution of relative invariants:
$$
\langle ~\alpha ~\rangle^{X(S)}_\beta = \sum_{\eta = (\Gamma_1, \Gamma_2), I} C_\eta\,\langle ~\alpha_1\mid \mu, e^I~\rangle^{\bullet (X(S), X(D))}_{\Gamma_1}  \langle ~\mu, e_I \mid \alpha_2~ \rangle^{\bullet (X(P), X(D))}_{\Gamma_2}.
$$
It gives rise to a system of linear equations where the relative invariants for the $\Bbb P^1$ bundle $X(P) \to X(D)$ are treated as \emph{coefficients of the system}. 
To handle these coefficients,  fiberwise localization on the $\Bbb P^1$ bundle is used and that inevitably introduces the descendants and, as noted above, breaks $\T$-invariance in the induction. 

While we employ some ideas from \cite{MP}, 
the arguments there have to be considerably modified to be useful for our purpose. 
In particular, we substitute \emph{ancestors} for descendants in key steps of degeneration analysis. 
Furthermore, the localization was replaced by more complex degeneration argument and the \emph{strong virtual pushforward property} developed by Cristina Manolache in \cite{vpf}.
These form the basis of inversion of degeneration arguments, and of the proof of (iii).

\begin{remark}
Indeed, the vital role played by ancestors in the study of Crepant Transformation Conjecture and the $K$-equivalence Conjectures had been first advocated and studied in \cite{ILLW, ypL} where the invariance of higher genus Gromov--Witten theory was established for \emph{simple} ordinary flops.
\end{remark}

\begin{remark}
Even though the theory of \emph{algebraic cobordism} \cite{LP3, mLP}
is not explicitly utilized, it inspires many ideas in this paper.
See Section~\ref{s:8.2} for some comments.
The complexity of the arguments could have been reduced 
and results greatly generalized had the
ideas inspired by algebraic cobordism been applicable.
\end{remark}

\subsection{Outline of the contents}
The sections of this paper are arranged in its logical order towards the proof of Theorem \ref{t:main}. 
The order of presentation is different from that described in the strategy outlined above. 
Thus we would like to briefly describe the contents in order for the readers to quickly locate necessary details for each step.   

In Section~1, we give some basic definitions and recall some facts on ordinary flops and Gromov--Witten theory.
In particular we define, in Section~\ref{s:1.3}, the notion of $\T$-invariance for a projective smooth variety $S$,
or rather a triple $(S, F, F')$ with $F$ and $F'$ two vector bundles on $S$ of equal rank.
We also recall, in Proposition~\ref{p:1.1.1}, that the proof of Theorem~\ref{t:main} can be reduced to that of ``local models'', a notion defined in Section~\ref{s:1.1}.
The degeneration formula and the product formula for relative invariants are briefly described in Sections~\ref{deg-formula} and \ref{s:1.7}.
As explained above, we our goal is to show a quantum splitting principle.
In Section~\ref{splitting}, we explain how splitting a bundle is done by birational modification and deformation.
The entire paper can be considered as a verification that this simple proceedure can be carried out while keeping $\T$-invariance in all steps.

Sections 2 to 4 are technical results needed for our proof, and can be omitted in the first reading by assuming the results. 
In Section~2, we recall and prove some general results about the \emph{fiber integrals}. 
Section~3 is devoted to some structural results relating 
invariants of a $\bp^1$ bundle with those of its base, without invoking fiberwise localization. In particular, we prove 
a \emph{strong virtual pushforward property} between various relative moduli spaces with $\bp^1$ bundles considered as rigid or non-rigid targets.
The deformation invariance of $\T$-invariance is established in Section~4.

In Sections~5 and 6, $\T$-invariance for ancestor invariants of $\bp^1$ bundle is proved, assuming $\T$-invariance for the base. 
These results can be viewed as the ancestor version of the corresponding \emph{quantum Leray--Hirsch theorem} in genus zero. 
A splitting principle for vector bundles under birational modification (statement (ii) in Section \ref{strategy}) is described in Section \ref{splitting}. 
It is already used there to determine the fiber class type II invariants in terms of the \emph{simple flop case} and the \emph{classical cohomology rings on the base}. 
The non-fiber class type II invariants as well as the rubber invariants are then handled inductively by the strong virtual pushforward property and related results established in Section~3. 
A technically important result is the \emph{inversion of degeneration} argument for ancestors of $\Bbb P^1$ bundles in Section~6 (Proposition \ref{prop:inversion'}).

Section~7 (Proposition \ref{prop:inversion} and Theorem \ref{thm:rel2abs}) proves the statement (iii) mentioned in Section \ref{strategy}. Namely $\T$-invariance for $S$ and $D$ implies $\T$-invariance for the smooth pair $(S, D)$, assuming the $\T$-invariance of the related $\bp^1$ bundles proved in Section~6. 

The last section (Section~8) concludes the proof by reducing to the case of split bundles, whose proof was done in Part II \cite{LLW2}. 
The proof consists of a refined induction procedure on ordinary flops constructed out of the triple $(S, F, F')$ such that $S$ is of the form 
$$S = \Bbb P_T(N_1) \times_T \cdots \times_T \Bbb P_T(N_k)$$ 
with $\{N_i\}$ a finite collection of vector bundles over $T$ and \emph{$F$ and $F'$ are pullbacks from $T$}. 
We proceed by induction on the dimensional of the base $T$. 

The \emph{topological recursion relation} (TRR) for the ancestors in genus zero relative Gromov--Witten theory, whose precise formula is not used in the main text,
is discussed in the appendix.

\begin{remark} \label{r:0.3.1}
By the  comparison results proved in \cite{ACW} and \cite{AMW}, the five different models of genus zero relative invariants are equivalent. 
Although we have employed primarily the logarithmic approach to relative Gromov--Witten invariants \cite{AC, GS},  
results concerning the genus zero relative Gromov--Witten invariants can also be deduced using, for instance, the orbifold approach of Cadman.

In this paper, we freely switch between the usual notations of the relative Gromov--Witten and the log Gromov--Witten perspectives. It is certainly possible to formulate everything in one perspective only. However, the available literature is scarce in this area and it is often easier to make connection with the existing literature by discussing one thing in the relative notation (e.g., the degeneration formula) and another in the log notation (e.g., the perfect obstruction theory).
\end{remark}

\section{Notations and basic facts} \label{s:1}
In this section, we review notations and facts on ordinary flops from Part I and II \cite{LLW1, LLW2}, including the notion of $\T$-invariance on generating functions of GW invariants and the degeneration formula. We also introduce the partial ordering on weighted partitions for relative invariants.

\subsection{Ordinary flops} \label{s:1.1}

\subsubsection{Local models of ordinary flops}
Given a triple $(S, F, F')$, where $S$ is a smooth projective variety over $\cc$ and $F, F'$ two vector bundles of rank $r+1$, we build the local model of ordinary flop $f: X \dasharrow X'$ as follows:
\begin{itemize}

\item $Z = Z(S,F,F')= \bp_S(F)$. There is a natural projection 
\[\phib_S: Z \to S.\]

\item $X = X(S, F, F')= \bp_{Z}(\phib_S^*F' \otimes \oo_Z(-1)\oplus \oo)$ with projection 
\[ p_S:X \to Z.\]

\item Denote by $i_S: Z \to X$ the inclusion which identifies $Z$ as the zero section of $\phib_S^*F' \otimes \oo_Z(-1)$.

\item Using the projective bundle structure $p_S$, we define $\xi=c_1(\mathscr{O}_X(1))$. Similarly, $h = c_1(\mathscr{O}_Z(1))$ for $\phib_S$. $h$ is understood either as a class in $Z$ or its pullback to $X$.

\item (Leray--Hirsch) $H^*(X)$ is generated by $h, \xi$ and (pull-back of) $H^*(S)$. $\xi$ can be represented by the \emph{infinity divisor} disjoint from $Z$.

\item $Z' = Z'(S, F, F')$, $\phib'_{S}$, $X' = X'(S, F, F')$, $p'_S$, $i'_S, h'$ and $\xi'$ are defined as above by switching the roles of $F$ and $F'$. For example,
$$X'(S, F, F') := X(S, F', F).$$

\item
Let $Y = Y(S, F, F')$ be the common blow-up
$$
Y =\bl_Z(X)=\bl_{Z'}(X')
$$ 
with blowdown maps  $\phi_S:Y \to X$ and  $\phi'_S:Y \to X'$.

\item The local model for the ordinary flop associated to $(S,F,F')$ is the birational map
\[f_S: \phi'_S \circ \phi^{-1}_S: X \dashrightarrow Y \to X'.\]

\item The correspondence
$$
\T_S = [\bar\Gamma_{f_S}]_* = {\phi'_S}_{*} \circ \phi_S^* : H^*(X) \to H^*(X')
$$ 
is an isomorphism preserving the Poincar\'e pairing.
\end{itemize}

If the triple $(S, F, F')$ is clear from the context,
we will use the shorthand notation $Z(S), X(S)$.
The letter $Z, Y$ will be used \emph{locally} to stand for other objects, 
but their meaning will be made clear in the context.

When $S$ is a point $\Spec \cc$, we will use $X(\cc)$  to denote 
$$
X(\Spec \cc, \mathscr{O}^{\oplus (r+1)}, \mathscr{O}^{\oplus (r+1)}).
$$

\subsubsection{General ordinary flops}

An ordinary flop $f: X \dashrightarrow X'$ is a birational map which is locally isomorphic to the local models in the neighborhood of the exceptional loci. We summarize these in the following commutative diagram
\begin{equation*} 
\xymatrix{ &
E\;\ar[ld]_>>>>{\,}\ar[rd]|\hole^>>>>{\,}
\ar@{^{(}->}[r]^j &
Y\;\ar[ld]_>>>>{\,}\ar[rd]^>>>>{\,}\\
Z\;\ar[rd]_>>>>>>>>{\phib}\ar@{^{(}->}[r]^i &
X\;\ar[rd]_>>>>>>>>>>{\psi} &
Z'\;\ar[ld]|\hole^>>>>>>>>>>{\phib'} \ar@{^{(}->}[r]^{i'} &
X'\;\ar[ld]^>>>>>>>>{\psi'}\\
&S\;\ar@{^{(}->}[r]^{j'}&\overline{X}& }
\end{equation*}
Note that $\psi$ and $\psi'$ are flopping contractions which contract $Z$ and $Z'$ to $S$ and $S'$ in $\bar{X}$, which is singular. (We will not use $\bar{X}$ except its existence, which guarantees the projectivity of $X'$ from that of $X$ and $\bar X$.)

There are two main differences between a general ordinary flop and its local model.
Obviously, the latter allows simpler topological and geometric description. More importantly for us, from the functorial perspective, the local models have an additional smooth morphism from $X$ to $Z$, and both $X \to Z$, $X \to S$ are fiber bundles with fibers toric manifolds. However, these toric bundles do not come from toric construction. 
Thus, the localization technique does not apply.

\begin{prop} [{\cite[Proposition 3.3]{LLW1}}] \label{p:1.1.1}
To prove Theorem~\ref{t:main}, it is enough to prove the corresponding statements for the local models.
\end{prop}

\emph{We will henceforth assume $X$ and $X'$ are the local models.}

\subsection{Generating functions of Gromov--Witten invariants} \label{s:1.2}

\subsubsection{Curve classes}
We first discuss the notation of curve classes in our generating functions. 

Given $X(S) \to Z(S) \to S$, denote by $\gamma$ the curve class of the fiber of $X(S) \to Z(S)$, and $\ell$ the curve class of the fiber of $Z(S) \to S$. When $S$ is a point, 
\[
N_1(X(S)) \simeq \zz \gamma \oplus \zz \ell.
\] 
If a curve class is effective, its $\gamma$ coefficient is non negative. Similarly we have $\gamma'$ and $\ell'$ for $X'(S) \to Z'(S) \to S$.

Given a curve class $\beta_S \in H_2(S,\mathbb{Z})$ and integer $d$, we say $\beta \in H_2(X,\mathbb{Z})$ \emph{belongs to} $(\beta_S, d)$ if $\int_\beta \xi = d$, and $\beta$ projects to $\beta_S$ under $X(S)\to S$. If $\beta$ belongs to $(\beta_S, d)$, then any other element belonging to $(\beta_S, d)$ is of the form $\beta+ k\ell$ for some integer $k$. Straightforward modification of the above definition applies to $X'$. (Notice that $d$ was called $d_2$ in Part I, II \cite{LLW1, LLW2}.)

We will refer to $(\beta_S, d)$ as \emph{curve classes (modulo the extremal ray)}. $(\beta_S, d) = (0, 0)$ are the \emph{extremal} classes. Other classes are called \emph{non-extremal}.

It is proved in Part I \cite[Lemma 3.5]{LLW1}, that  $\beta$ belongs to $(\beta_S, d)$ if and only if $\T(\beta)$ belongs to $(\beta_S, d)$.

\subsubsection{Gromov--Witten invariants} \label{GWI}
We consider only \emph{genus zero} GW invariants.

Let $\mbar_{0,n}(X, \beta)$ be the moduli stack of stable maps from $n$-pointed, genus zero pre stable curves to $X$ with curve class $\beta \in H_2(X)$.

We have evaluation maps
\begin{equation} \label{e:ev}
\ev_j: \mbar_{0, n}(X, \beta) \to X
\end{equation}
for each $1 \le j \le n$. When $n \ge 3$, there is a stabilization map 
\[ 
\st: \mbar_{0, n}(X, \beta) \to \mbar_{0,n}.
\]
The notation 
\[
\Big\langle~
\prod_{j=1}^n \taub_{k_j}(\alpha_j) ~\Big\rangle_\beta^X =
\int_{[\mbar_{0,n} (X, \beta)]^{\vir}} \prod_{j}\ev_j^*(\alpha_j)
\st^*(\prod_j \psi_j^{k_j})
\]
stands for the \emph{ancestor invariant},
where $k_j$'s are nonnegative integers,  $\alpha_j \in H^*(X)$, and $\psi_j \in H^2(\mbar_{0, n})$ the cotangent 
class for the marked point labeled by $j$.

We also define the cycle
\begin{equation} \label{cycle}
\llangle~ \prod_{j=1}^n \taub_{k_j}(\alpha_j)  ~\rrangle_\beta^X
 :=  [\mbar_{0,n} (X, \beta)]^{\vir} \cap \prod_{j}\ev_j^*(\alpha_j)
   \cap  \st^*(\prod_j \psi_j^{k_j}).
\end{equation}
Note that 
\[
\st_*\Big(\llangle~ \prod_{j=1}^n \taub_{k_j}(\alpha_j)  ~\rrangle_\beta^X \Big)
\] is  a nonzero cycle on $\mbar_{0,n}$ only if
\[
n-3 \ge \sum k_j.
\]

Similar notations are used for relative invariants. Let $B$ be a smooth divisor in $X$, $\mu$ a partition of $\int_\beta B$, and $\mbar_{0,n+m} (X, B;\beta,\mu)$ the moduli stack of relative stable maps to $(X, B)$ with indicated discrete contact data. Here $m = l(\mu)$ is the number of contact points in $B$.

When $n + m \ge 3$, by abusing  notation, we also use $\st$ to denote the stabilization map 
\begin{equation} \label{e:st}
 \st: \mbar_{0, n+m}(X, B;\beta,\mu) \to \mbar_{0, n+m} .
\end{equation}
In addition to the evaluation maps \eqref{e:ev} for the internal marked points, we have the evaluation maps for the relative marked point
\begin{equation} \label{e:rel-ev}
 \ev^B_i =: \ev_i : \mbar_{0, n+m}(X, B;\beta,\mu) \to B, \quad i= n+1, \ldots, n+m.
\end{equation}
\emph{We will abused the notation and dropped the superscript $B$ when there is no confusion.}

For $\alpha_j \in H^*(X), \delta_i \in H^*(B)$, non-negative integers $\{k_j\}$ and $\{l_i\}$, we use 
\[
\Big\langle~
\prod_{j=1}^n \taub_{k_j}(\alpha_j) \mid \prod_{i=1}^m\taub_{l_i}(\delta_i)
~\Big\rangle_\beta^{(X, B)} 
\] 
to represent the relative invariant 
\[
\int_{[\mbar_{0,n+m} (X, B; \beta,\mu)]^{\vir}} \prod_{j}\ev_j^*(\alpha_j) 
\prod_{i}\ev_i^*(\delta_i)
\st^*(\prod_j \psi_j^{k_j} \prod_i \psi_i^{l_i}),
\]
and 
\begin{equation} \label{rel-cycle}
\llangle~ \prod_{j=1}^n \taub_{k_j}(\alpha_j) \mid \prod_{i=1}^m\taub_{l_i}(\delta_i) ~\rrangle_{\beta, \mu}^{(X, B)}
\end{equation}
the cycle
\[
[\mbar_{0,n+m} (X, B;\beta,\mu)]^{\vir} \cap \prod_{j}\ev_j^*(\alpha_j) \prod_i \ev_i^*(\delta_i)
   \cap \st^*( \prod_j \psi_j^{k_j} \prod_i \psi_i^{l_i}).
\]

The following shorthand notation for relative invariants will also be used:
\[
 \langle~ \mu, \gamma \mid\alpha  ~\rangle_\beta^{(X, B)}, \quad
\langle~ \alpha \mid\vec{\mu} ~\rangle_\beta^{(X, B)}, \quad
 \langle~  \vec{\mu}\mid\alpha  ~\rangle_\beta^{(X, B)}, \quad
 \langle~  \vec{\mu}\mid\alpha \mid\vec{\nu} ~\rangle_\beta^{(X, B)}.
 \]
The first three are all type I invariants and the last one is referred as a type II invariant, which is used when $B = B_0 \coprod B_\infty$ is a disjoint union of two divisors. Here $\alpha$ denote insertions for non-relative marked points, $\gamma$ insertions for relative marked points, and $\vec{\mu}$ a \emph{weighted partition} defined in Section~\ref{s:1.5}, which combines the relative profile $\mu$ and relative insertions. The symbols $\mu,\nu$ or $\lambda$ denote partitions, and $\vec{\mu} = (\mu, \gamma)$ etc.~weighted partitions.

\subsubsection{Generating functions}

A superscript is used to indicate the target and a subscript for the curve class modulo the extremal ray. 
Denote by $\langle~ \alpha ~\rangle ^{X(S)}_{(\beta_S,  d)}$ the generating function 
\[
 \langle~ \alpha ~\rangle ^{X(S)}_{(\beta_S,  d)} = 
 \sum_{ \beta \in (\beta_S, \,d) }\langle~ \alpha ~\rangle ^{X(S)}_\beta q^\beta,
\]
where $\langle~ \alpha ~\rangle ^{X(S)}_\beta$ is the genus zero GW invariant of $X(S)$ with class $\beta$ and insertions $\alpha$. Similar generating functions are used for relative invariants. Given a smooth divisor $D \hookrightarrow S$, \emph{not necessarily connected}, 
\[
  \langle~ \alpha \mid\vec{\nu} ~\rangle ^{(X(S),\, X(D))}_{(\beta_S, \,d)}
 = \sum_{ \beta \in (\beta_S, \,d) } 
    \langle~ \alpha \mid\vec{\nu} ~\rangle ^{(X(S),\, X(D))}_{\beta} q^\beta.
\]
By abuse of the language, we sometimes refer to a generating function as \emph{invariants} when confusion is unlikely to arise.

\subsection{Definition of $\T$-invariance} \label{s:1.3}
Given a triple $(S, F, F')$, we say that \emph{$\T$-invariance holds for $(S,F,F')$}
with curve classes $(\beta_S,d)$ (modulo extremal ray) if for any ancestor insertions $\omega$,
\begin{equation} \label{e:1.3.1}
\T \Big(
\langle~ \omega ~\rangle ^{X(S)}_{(\beta_S, \,d)}\Big)
= \langle~ 
\T(\omega) ~\rangle ^{X'(S)}_{(\beta_S, \,d)}
\end{equation}
after analytic continuation. If $\omega$ is a primary insertion, we require the number of insertions $|\!|\omega|\!| \ge 3$. When the context is clear, we sometimes omit the bundles and say ``$\T$-invariance holds for $S$''.

We explicate the meaning of analytic continuation. As explained earlier, both sides of \eqref{e:1.3.1} are generating functions which, a priori, are formal power series in $q^{\ell}$ and $q^{\ell'}$ respectively. When we say \emph{the equality holds in \eqref{e:1.3.1} after analytic continuation}, we mean that the LHS is an analytic function in $q^{\ell}$,
the RHS is an analytic function in $q^{\ell'}$, and they are equal \emph{as analytic functions} after the identification
\[
 \T q^{\ell} = q^{\T \ell} = (q^{\ell'})^{-1}.
\]
The last equality follows from $\T (\ell) = -\ell'$, which was shown in \cite{LLW}.

We also note that the condition $|\!|\omega|\!| \ge 3$ is necessary for primary insertions. (Ancestors by definition satisfy the condition.) Counterexamples of $\T$-invariance were given in \cite{LLW} for $|\!|\omega|\!| < 3$.

Given a smooth divisor $D \hookrightarrow S$, we say $\T$-invariance holds for $(S,D)$ (bundles omitted) if whenever $|\!|\omega|\!| + l(\vec{\nu}) \ge 3$,
\[
\T\left(
\langle~ \omega \mid\vec{\nu}~\rangle ^{(X(S), \,X(D))}_{(\beta_S, \,d)}\right)
=
\langle~ 
\T(\omega) \mid\T(\vec{\nu})~\rangle ^{(X'(S), \,X'(D))}_{(\beta_S, \,d)}.
\] 
Note that by the product rule, $\T$-invariance for disconnected domain curve also holds when there is no contracted component and each component contains at least 3 insertions.

$\T$-invariance for non-rigid targets, also termed \emph{rubber targets} 
\cite{MP}, can be formulated similarly. The definition of relative invariants for \emph{rubbers} can be found in \cite[Section~2.4]{GV}.

\begin{remark} \label{r:1.3.1}
In fact, for non-extremal $(\beta_S, \,d)$, the condition $|\!|\omega|\!| \ge 3$
is not necessary. For primary invariants, the divisor equation allows us to create extra insertions, upholding $\T$-invariance for absolute and relative invariants for few insertions. More precisely, for $\beta_S \neq 0$, pick an ample divisor $H$ on $S$. Then
\[
 \langle~ \omega ~\rangle ^{X(S)}_{(\beta_S, \,d)} = 
  \frac{1}{\int_{\beta_S} H}
 \langle~ \phi^* H, \omega ~\rangle ^{X(S)}_{(\beta_S, \,d)}.
\]
Since $H$ is a pullback class, 
\[
 \int_{\beta} \phi^*H = \int_{\beta_S} H.
\] 
The key point is that the intersection number of $H$ and any $\beta \in (\beta_S,d)$ is constant and the generating function remains unchanged. By ampleness of $H$ on $S$, $\int_{\beta_S} H \neq 0$. This way, the $\T$-invariance of generating functions with fewer insertions is implied by that with arbitrary number of insertions. If $d \neq 0$, replace $H$ by $\xi$. Since $\T(\xi) =\xi'$, the same argument works.
\end{remark}

\subsection{Some geometric properties of the local models}

\subsubsection{Functoriality}
We view $Z$, $X$, $Z'$, $X'$ as functors over the category of smooth projective schemes (over $S$). Given a map $g: T \to S$, we have the pullback triple $(T, g^*F, g^*F')$. Any relevant diagram constructed from $(T, g^*F, g^*F')$ can be obtained by base change from $T \to S$. For instance, we have
\[
\xymatrix{
X(T) \ar[r]^{p_T} \ar[d]^{g_X} & Z(T) \ar[r]^{\phib_T}\ar[d]^{g_Z} & T\ar[d]^{g}\\
X(S) \ar[r]^{p_S} & Z(S)\ar[r]^{\phib_S} & S.
}\]
Later, for simplicity, we will denote by $|_{X(S)}$ the pull back under $X(S) \to S$. 

Due to the motivic property of $\T$, it is compatible with proper pushforward and pullback. Denote $\T_S$ the correspondence $\T$ modeled on $S$. Given a map $g: T \to S$, we have 
\[
  \T_S \circ {g_X}_*  = {g_{X'}}_* \circ \T_T, \qquad  \T_T \circ g_X^*  
 = g_{X'}^* \circ \T_S.
\]

\subsubsection{$\T$ in terms of a basis and its dual basis.}
Let $\{T_i \}$ be a basis of $H^*(S)$,  $\{ T_i^{\vee} \}$ its dual basis, then 
$$
\{  {T_i}h^j\xi^k \mid 0 \le  j\le r, \, 0\le k\le r+1 \}
$$ 
is a basis of $H^*(X(S))$, with dual basis $\{T_i^{\vee}H_{r-j}\Theta_{r+1-k}\}$. See \cite[Lemma 1.4, Lemma 2.1]{LLW1}. 

We do not need the precise formulas for $H_i$ or $\Theta_j$ in this paper. We only remark that $H_i$'s (resp.~$\Theta_j$'s) are the Chern classes of certain quotient bundles of rank $r$ over $S$ (resp.~rank $r + 1$ over $Z$). In particular they are defined only for $0 \le i \le r$ (resp.~$0 \le j \le r + 1$). For notational convenience, we define $H_i = 0 = \Theta_j$ for $i$, $j$ outside the above range. Similar remarks apply to $H'_i$ and $\Theta'_j$ on the $X'$ side as well.

\begin{prop}[{\cite{LLW1}}] \label{prop:homo}
$\T_S$  is $H^*(S)$-linear and for $j \le r$ or $k \ne 0$,
$$
\T(h^j\xi^k) = (\xi' -h')^j {\xi'}^k.
$$
\end{prop}

Since $\T$ preserves the Poincar\'e pairing, this implies
\begin{cor} \label{cor:hom}
There are constants $C(i, j, k)$ such that 
$$
\T(H_j\Theta_k)=\sum_{m=0}^{j+k} C(j,k,m)H'_m\Theta'_{j+k-m}.
$$
\end{cor}

\begin{proof}
Let ${\rm pt}$ be the point class in $H^*(S)$. Then
$$
\T ({\rm pt}.h^j\xi^k) = {\rm pt}.(\xi' -h')^j {\xi'}^k = \sum_{m = 0}^j C^j_m (-1)^m {\rm pt}.{h'}^{j - m}{\xi'}^{m + k}. 
$$
The Chern polynomial equations allow us to rewrite the sum as
$$
\sum_{m = 0}^{r} C(j, k, m)\, {\rm pt}.{h'}^m {\xi'}^{j + k - m},
$$
where $C(j , k, m) = 0$ if $j + k - m < 0$ or $j + k - m > r + 1$. 
The result follows by Poincar\'e duality.
\end{proof}

\subsubsection{$\T$-homogeneous basis}
Let 
$$
G = GL_{r+1} \times GL_{r+1},
$$ 
then the pair of vector bundles $(F, F')$ determines a map 
\[S \to BG = BGL_{r+1} \times BGL_{r+1}  ,\]
where $BG$ is the \emph{classifying stack of $G$ bundles}. We have a  cartesian diagram
\[
\xymatrix{
X(S) \ar[r]\ar[d] &  [X(\cc)/G]\ar[d] \\
S \ar[r] & BG.
}
\]
Let $\{T_i\}$ be a basis of $H^*(S)$. Let  $\{ V_m\}$ be classes on $[X(\cc)/G]$ such that the pullbacks of $V_m$'s to $X(\cc)$ via $X(\cc) \to [X(\cc)/G]$ form a basis of $H^*(X(\cc))$. Then $\{ T_iV_m\}$ form a basis of $H^*(X(S))$ by Leray--Hirsch theorem. 
Here $T_i, V_m$ are viewed as classes on $X(S)$ via $X(S) \to S$ and $X(S) \to [X(\cc)/G]$ respectively. Note that the classes ``$h, \xi, H_j$ and  $\Theta_k$'' are indeed defined on $[X(\cc)/G]$, and their pullbacks are the classes  $h, \xi, H_j$ and  $\Theta_k$ in $H^*(X(S))$.

\begin{defn} \label{def:hom}
A $\T$-homogeneous basis for $X(S,F,F')$ is a basis of its cohomology of the form $\{ T_iV_m \}$, where $T_i$ are homogeneous classes on S, $V_m$ are homogeneous classes on $[X(\cc)/G]$. We further require
\begin{enumerate}
\item $\{T_i\}$ is a basis of $H^*(S)$. 

\item $V_m$ is a linear combination of $\{ h^j\xi^k\}_{\{0 \le j \le r,\, 0\le k \le r+1\}}$, and the pullbacks of $V_m$'s to $X(\cc)$ form a basis of $H^*(X(\cc))$.
\end{enumerate}
In this case, we say $\{T_iV_m\}$ is compatible with $\{T_i\}$.

As $X'(S, F, F') = X(S, F', F)$. A $\T$-homogeneous basis for $X'(S, F, F')$ is by definition a  $\T$-homogeneous basis for $X(S, F', F)$.
\end{defn}

\begin{lemma}
If $\{T_iV_m\}$ is a $\T$-homogeneous basis for $X(S)$, then $\{T_i\T(V_m)\}$ is a $\T$-homogeneous basis for $X'(S)$.
\end{lemma}

\begin{proof}
This follows from Proposition \ref{prop:homo}.
\end{proof}

The significance of requiring $V_m$ to be linear combinations of $\{ h^j\xi^k\}$ is stated in the following lemma.

\begin{lemma}
If $\{T_iV_m\}$ is a $\T$-homogeneous basis for $X(S)$, then its dual basis is of the form $\{T_i^{\vee}W_m\}$. Here $\{T_i^{\vee}\}$ is the dual  basis of $\{T_i\}$, $W_m$ is a linear combination of $\{H_{r-j}\Theta_{r+1-k}\}$, and when pulled back to $X(\cc)$,
$\{W_m\}$ is the dual basis of $\{V_m\}$.
\end{lemma}

\begin{proof}
This follows from Corollary \ref{cor:hom}.
\end{proof}

By abuse of notation, we will write the dual basis of $\{T_iV_m\}$ as 
\[
\{T_i^{\vee}V_m^{\vee}\},
\] 
although $V_m^{\vee}$ is \emph{not} the dual of $V_m$ in $[X(\cc)/G]$.

The abstraction to  $\T$-homogeneous basis serves to streamline arguments involving the choice of a basis. In practice, we only need the particular $\T$-homogeneous basis $\{T_ih^j\xi^k \}$ for $X(S)$ and $\{T_i\T(h^j\xi^k)\}$ for $X'(S)$.

\subsection{Weighted partitions and partial orderings} \label{s:1.5}

\subsubsection{Weighted partitions} \label{s:3.1.2} 

As above, let $(S, D)$ be a smooth pair. We will order relative invariants with respect to a chosen $\T$-homogeneous basis for $X(D)$.

Let $\{ \delta_aV_m \}$ be a $\T$-homogeneous basis of $H^*(X(D))$ compatible with $\{\delta_a\}$. Consider generating functions of the form 
\begin{equation} \label{e:1.5.1}
\langle~ \alpha \mid\{ (\mu_i,\delta_{a_i}V_{m_i})\}~\rangle^{(X(S),\, X(D))}_{(\beta_S,\, d)},
\end{equation}
where $\alpha$ denotes a sequence of primary insertions with length $|\!|\alpha|\!|$.

Following \cite{MP}, we call the contact data
\[
\vec{\mu} := \{ ( \mu_1 , \delta_{a_1}V_{m_1}), \cdots, (\mu_{l(\mu)} , \delta_{a_{l(\mu)}} V_{m_{l(\mu)}} ) \}
\]
a \emph{(cohomology) weighted partition}. Here $l(\mu)$ is the length of the relative profile $\mu$, which is called the number of contact points $\rho$ in \cite{LLW, LLW1, LLW2}. We emphasis that the cohomology classes in a weighted partition
are always chosen from a $\T$-homogeneous basis.

Some notations will be used throughout the paper. For a weighted partition $\vec{\mu} = \{ ( \mu_i , \delta_{a_i}V_{m_i}) \}$, define 
\begin{equation} \label{k-num}
 k(\vec{\mu}) :=\max\{0, \sum_{\{i \,\mid\, \mu_i >1\}}\mu_i - \on{Id}(\vec{\mu}) \},
\end{equation}
where 
\[
 \on{Id}(\vec{\mu}) := \#\{i \mid (\mu_i, \deg\delta_{a_i},  \deg V_{m_i}) = (1, 0, 0) \}.
 \]
$k(\vec{\mu})$ is the number of extra divisorial insertions to ensure the existence of $\mbar_{0,n+m}$ in \eqref{e:st}
and hence that of the corresponding ancestors.
It will be used in Sections~6 and 7. See, for example,  Proposition~\ref{prop:inversion'}.

We use $\deg_D(\vec{\mu})$ to denote the total degree of the relative insertions
\[
 \deg_D(\vec{\mu}) :=\sum \deg{\delta_{a_i}}.
\]

\subsubsection{Partial orderings} \label{s:1.5.2}

This is a variant of the partial ordering in \cite{MP}.

Given $\vec \mu$, the set $\{( \mu_i, \deg \delta_{a_i},  \deg V_{m_i}) \in \Bbb N \times \Bbb Z_{\ge 0}^2 \mid i = 1, \ldots, l(\vec \mu)\}$ can be rearranged in the decreasing lexicographical order. For two such sets associated to $\vec \mu$ and $\vec \mu'$ with the same $l(\vec \mu) = l(\vec \mu')$, we say that 
\[
 \{( \mu_i, \deg \delta_{a_i},  \deg V_{m_i})\} >_l
 \{( \mu'_i, \deg \delta_{a'_i},  \deg V_{m_i'})\}
\]
if, after placing them in the decreasing lexicographical order, the first triple for which they differ is greater for $\{( \mu_i, \deg \delta_{a_i},  \deg V_{m_i}) \}$.

Now we define a \emph{partial ordering $\succ$ on the weighted partition $\vec{\mu}$}
by lexicographic order on the triple
$$
(\deg_D(\vec{\mu}), l(\mu), \{( \mu_i, \deg \delta_{a_i},  \deg V_{m_i})\} ),
$$
where 
\begin{itemize}
\item[(i)] For $\deg_D(\vec{\mu})$, smaller number corresponds to higher order in $\succ$,
\item[(ii)] For $l(\mu)$, larger number corresponds to higher order in $\succ$,
\item[(iii)] For $\{ ( \mu_i, \deg \delta_{a_i},  \deg V_{m_i}) \}$, lower order in $>_l$ corresponds to higher order in $\succ$.
\end{itemize}

Based on it, \emph{generating functions} of the above form \eqref{e:1.5.1} can be partially ordered lexicographically by the triple 
$$
((\beta_S, d), |\!|\alpha|\!|, \vec{\mu}):
$$
\begin{enumerate}
\item Curve class $(\beta_S, \,d)$: $(\beta_S, \,d) > (\beta_S', d')$ if $\beta_S - \beta_S' > 0$, or $\beta_S =\beta_S'$ and $d > d'$. 
\item Number of internal insertions $|\!| \alpha|\!|$: more insertions corresponds to higher order.
\item $\vec{\mu} = \{ ( \mu_i , \delta_{a_i}V_{m_i}) \}$ are ordered by $\succ$ defined above.
\end{enumerate}


A partial ordering on a set satisfies the \emph{DCC} if any descending chain has only finite length. 
DCC is an essential condition for the induction process. 
The partial ordering $\succ$ on the set of all weighted partitions does not satisfy DCC. 
However, the partial ordering on the set of generating functions does, as the following lemma shows.

\begin{lemma}
The partial ordering on the generating functions of (primary) relative GW invariants (\ref{e:1.5.1}) associated to a fixed triple $(S, F, F')$
satisfies the DCC. 
\end{lemma}

\begin{proof}
Given $((\beta_S, d), |\!|\alpha|\!|, \vec{\mu})$, the virtual dimension of the moduli space of relative stable maps is given by
$$
 \int_{(\beta_S, d)} c_1(X(S)) + |\!|\alpha|\!| + l(\vec \mu) - \sum_{i = 1}^{l(\vec \mu)} \mu_i.
$$ 
The first term depends only on $(\beta_S, d)$ since $\int_\ell c_1(X(S)) = 0$. In order to get non-trivial invariants it must agree with
$$
\sum_{i = 1}^{|\!|\alpha|\!|} \deg \alpha_i + \deg_D \vec \mu + \deg V_{m_i}.
$$
For a fixed $\beta_S^{\circ}$, the set 
$$ \{(\beta_S^{\circ}, d') \mid (\beta_S^{\circ}, d') < (\beta_S^{\circ}, d^{\circ}) \}$$
is clearly finite. 
Furthermore, the set 
$$ \{ \beta'_S \mid \beta'_S < \beta_S \}$$
is also finite.
Thus in any descending chain, $(\beta_S, d)$ is stabilized in finite steps. 
Similarly the number of insertions $|\!| \alpha |\!|$ stabilizes in finite steps.  

$\deg_D \vec \mu$ may increase in a descending chain. 
However, since 
$$l(\vec \mu) - \sum_{i = 1}^{l(\vec \mu)} \mu_i \le 0,$$ 
the virtual dimension count gives an upper bound for it and then 
$\deg_D \vec \mu$ stabilizes. 
Then the number of contact points $l(\vec \mu)$ stabilizes as well. It is clear that the remaining choices for $\vec \mu$ form a finite set.
\end{proof}

\subsection{Degeneration formula} \label{deg-formula}

We will apply the degeneration formula to a family obtained from deformation to the normal cone.

\subsubsection{Deformation to the normal cone}\label{subsec:blowup}

Given a smooth pair $(S, Z)$ with $Z$ is a smooth closed subvariety of $S$, we introduce the following notations. Let $N = N_{Z/S}$ be the normal bundle of $Z$ in $S$, $\tilde{S} =\bl_ZS$ be the blow up of $S$ along $Z$, $E=\bp_Z(N)$ the exceptional divisor in $\tilde{S}$, and $P=\bp_Z(N \oplus \oo)$. Let 
$$
p: P \to Z \to S, \qquad \phi: \tilde{S} \to S 
$$
be the natural morphisms. 

The deformation to the normal cone for the pair $(S, Z)$ is simply
$$ 
W =W(S,Z)=\bl_{Z \times \{0\}} (S \times \A^1).
$$ 
Then $W_t =S$ when $t \ne 0$ and $W_{t=0}$ is obtained by gluing $\tilde{S}$ and $P$ along $E$. It is easy to see that
$$
X(W(S, Z)) \to \A^1
$$ 
and $W(X(S), X(Z)) \to \A^1$ are isomorphic.

\subsubsection{Degeneration formula}

Applying the degeneration formula to the family $X(W) \to \A^1$, we get a degeneration of absolute invariants:
\begin{multline} \label{e:deg}
\langle~ \alpha ~\rangle_{(\beta_S, \,d)}^{X(S)} 
= \sum_{I } \sum_{\eta=(\Gamma_1, \Gamma_2)} C_\eta \times \\
{\phi_X}_*(\langle~ \alpha_1 \mid \mu , e^I~\rangle ^{\bullet (X(\tilde{S}), X(E))}_{\Gamma_1}) \cdot
{p_X}_*( \langle~ \mu, e_I  \mid \alpha_2 ~\rangle^{\bullet (X(P), X(E)) }_{\Gamma_2}).
\end{multline}
Here $\eta=(\Gamma_1,\Gamma_2)$ is a splitting of the discrete data. $C_\eta$ is a constant determined by $\eta$. $\Gamma_i$ specifies curve classes modulo extremal rays, 
non-relative and relative marked points, and contact orders of relative marked points encoded in a partition $\mu$ for each component. $I$ is an index set of $l(\mu)$ elements, $e^I \in H^*(X(E))^{\oplus l(\mu)}$, $e_I$ its dual with respect to some basis of
$H^*(X(E))$. $\alpha$ represents primitive insertions, $\alpha_i$ are the corresponding insertions specified by $\Gamma_i$.

As we are dealing with generating functions, we identify variables using $\phi_X$ and $p_X$ induced from $\phi: \tilde{S} \to S $ and $p: P  \to S$. For instance 
$$
{\phi_X}_*(q^{\beta}) = q^{{\phi_X}_*(\beta)}.
$$
Assume the total curve classes of $\Gamma_1$ is $(\beta_{\tilde{S}}, d_{\tilde{S}})$ and that of $\Gamma_2$ is $(\beta_{P}, d_P)$. The constrains on curve classes are
$$
\beta_S = \phi_*(\beta_{\tilde{S}}) + p_*(\beta_{P}),\qquad 
\int_{\beta_{\tilde{S}}} E = \int_{\beta_P} E, \qquad
d = d_{\tilde{S}} + d_P.
$$
Further analyzing these constraints on each connected component of $\Gamma_i$, we see that if $(\beta_S,\, d)$ is non-extremal, then any component specified by $\Gamma_1$ or $\Gamma_2$ has non-extremal curve classes.

We note that \eqref{e:deg} has a natural extension which allows the left hand side to be a relative invariant.
It comes in different variants, but all ``obvious'' extensions of \eqref{e:deg}.
The readers who are unfamiliar with them can consult \cite{MP} for example.

\subsection{A product formula for the relative invariants} \label{s:1.7}
We recall a product formula for relative invariants proved in \cite{LQ}.

Let $X$ and $Y$ be nonsingular projective varieties , 
and $D$ a a smooth divisor in $Y$. 
We further assume $H^1(Y)=0$, so a curve class of $X\times Y$ is of the form $(\beta_X, \beta_Y)$ 
where $\beta_X$ (resp.\ $\beta_Y$) is a curve class of $X$ (resp.\ $Y$).

The product formula is best formulated in terms of the \emph{Gromov--Witten correspondence}.
Let $\Gamma_X=(g,n+\rho, \beta_X)$. The Gromov--Witten correspondence
\[
 R_{\Gamma_X}: H^*(X)^{\otimes (n+\rho) }  \to H^*(\mbar_{g,n+\rho}),
\] 
is defined by
\[
  R_{\Gamma_X} (\alpha) := \operatorname{PD} \left( \st_* \left( ev_X^*(\alpha ) \cap [\mbar_{\Gamma_X} (X)]^{vir} \right) \right),
\]
where $\operatorname{PD}$ stands for the Poincar\'e duality and $\st$ and $\ev$ are defined in \eqref{e:st} and \eqref{e:ev}.

For the relative Gromov--Witten correspondence for the smooth pair $(Y,D)$,
there is a similarly defined Gromov--Witten correspondence
\[
R_{\Gamma_{(Y,D)}}: H^*(Y)^{\otimes n} \otimes H^*(D)^{\otimes \rho} \to H^*(\mbar_{g,n+\rho})
\]
defined as 
\[
  R_{\Gamma_{(Y, D)}} (\alpha'; \delta) :=
 \operatorname{PD} \left( \st_* \left( ev_Y^*(\alpha ) ev_D^*( \delta ) 
  \cap [\mbar_{\Gamma_Y}(Y,D)]^{vir} \right) \right),
\]
where, by a slight abuse of notation, 
$\Gamma_{(Y,D)}$ encodes all discrete data of the relative moduli.


\begin{thm}[{The product formula for $X \times (Y,D)$} {\cite[Corollary~3.1]{LQ}}] \label{t:product}
Let 
$$\Gamma_{X \times (Y,D)} =(g,n,(\beta_X,\beta_Y),\rho,\mu)$$ 
be the relative data for the product $X \times (Y,D)$. We have 
\[
 \begin{split}
  &R_{\Gamma_{X \times (Y,D)}}
   ((\alpha_1 \otimes \alpha'_1) \otimes ... \otimes  (
  (\alpha_n \otimes \alpha'_n));
   (\alpha_{n+1} \otimes \delta_1) \otimes ... \otimes ((\alpha_{n+\rho} \otimes \delta_\rho)) \\
 = &R_{\Gamma_X} (\alpha_1 \otimes ... \otimes  \alpha_{n+\rho}) 
   R_{\Gamma_{(Y, D)}} (\alpha'_1 \otimes ... \otimes \alpha'_n; 
   \delta_1\otimes ...\otimes \delta_\rho),
 \end{split}
\]
where $\alpha_i \in H^*(X)$, $\alpha'_i \in H^*(Y)$ and
$\delta_j \in H^*(D)$.
\end{thm}

\begin{remark}
(i) In this paper, only the special case with $g=0$ and $(Y,D) =(\mathbb{P}^1, \{ pt \})$ is used.

(ii) The above product formula can be reformulated in terms of more refined GW invariants,  
with additional insertions from arbitrary cycles in $\mbar_{g,n}$. 
For example, for any class $\gamma \in H^*(\mbar_{g,n})$,
\[
  \int_{\mbar_{g,n}} \st^*(\gamma) R_{\Gamma_X} (\alpha) = \langle \gamma, \st_* ( ev_X^*(\alpha ) \cap [\mbar_{\Gamma_X} (X)]^{vir} ) \rangle,
\]
where $\langle \cdot \rangle$ is the pairing between cohomology and homology of $\mbar_{g,n}$.
Since the Poincar\'e duality holds for $H^*(\mbar_{g,n})_{\mathbb{Q}}$, the above integral with arbitrary $\gamma$ gives $R_{\Gamma_X} (\alpha)$.
In genus zero, however, $H^*(\mbar_{0,n})$ is generated by $\psi$-classes and the refined GW invariants above are simply \emph{ancestors}.
Similar discussion applies to the relative invariants.
\end{remark}

\subsection{Splitting bundles} \label{splitting}
Our main task is to reduce the proof of $\T$-invariance of $(S,F,F')$
from the non-split bundles to split bundles.
Here we explain how the bundles can be split ``classically''.

\begin{lemma} \label{l:5.3.1}
Given a rank $r+1$ vector bundle $F \to S$,
there exists a sequence of blowing-ups on smooth centers such that
the pullback of $F$, denoted $\pi^*(F)$, admits a filtration of subbundles
\[
 0=F_0 \subset F_1 \subset \ldots \subset F_{r+1} = \pi^*(F),
\]
satisfying $\operatorname{rank}(F_{i+1}/F_i) =1$ for all $i$.
\end{lemma}

\begin{proof}
Consider the complete flag bundle over $S$
\[
  p: \Fl (F) \to S.
\]
By local triviality, $p$ admits a rational section $s$.
Resolving the rational map $S \dashrightarrow \Fl(T)$ by
a sequence of blowing-ups along smooth centers, one gets
\[ 
 \pi: \tilde{S} \to S
\]
such that $\pi^* \Fl (F)$ admits a section.
\end{proof}
We say that $\pi^* F$ \emph{admits complete flags} when the conclusion of
Lemma~\ref{l:5.3.1} holds. 

\begin{lemma} \label{l:deform}
Let $F \to S$ be a vector bundle admitting complete flags.
Then it can be deformed to $F_1 \to S$ such that $F_1$ is a split bundle.
\end{lemma}

\begin{proof}
The non-triviality of complete flags is governed by multiple extension classes.
There is a deformation of the bundles, with the base $S$ fixed,
sending all extension classes to zero.
\end{proof}

\section{Fiber integrals} \label{s:2}
\setcounter{subsection}{+1}

In this section, we assemble some results about fiber integrals. By fiber integrals we mean, after \cite{MP}, the GW invariants of $G$-principal bundles or their associated fiber bundles $p:E \to B$ with the curve class $\beta$ such that $p_*(\beta) =0$. This should not be confused with the similarly named \emph{fiber curve class (modulo the extremal rays)} in Section~\ref{s:5.2.2}.

We first recast results from \cite[Section~1.2]{MP} in the form needed for our purpose. This also allows us to deduce Proposition~\ref{prop:bir}. All results in this section are for fiber integrals only.

All the schemes we consider are smooth. Let $G$ be a group scheme over $\cc$, and $BG =[\Spec \cc / G]$ the classifying stack of $G$ bundles. Given a smooth map $f: B \to BG$ and a $G$-equivariant smooth pair $(F, D_F)$, where $D_F$ is a $G$-divisor in $F$, define the fiber bundle pairs $(E,D_E)$ over $B$ as the fiber product
$$
\xymatrix{
(E, D_E) \ar[r]^-{g}\ar[d]^p  &( [F/G], [D_F/G])\ar[d] \\
B \ar[r]^{f} &BG .
}
$$

Now we switch to the notation of \emph{log geometry} (see Remark~\ref{r:0.3.1} for justification). Let $E^{\dagger}$ (resp.\ $F^{\dagger}$) be the log scheme with the divisorial log structure determined by $D_E$ (resp.\ $D_F$). We have a cartesian diagram of log stacks
\[
\xymatrix{
E^{\dagger} \ar[r]^-{g}\ar[d]  & [F^{\dagger}/G]\ar[d]^{\pi} \\
B \ar[r]^{f} &BG,
}
\]
where $B$ and $BG$ are equipped with the trivial log structure. In particular $f$ is strict.

Given a $G$-invariant curve class $\beta $ in $F$, we have a cartesian diagram  between log stacks
\begin{equation} \label{e:2.1.1}
\xymatrix{
 \mbar_{0,n}(E^{\dagger}/B, \beta)\ar[r]^-{\bar{g}}\ar[d]
   &\mbar_{0,n}( [F^{\dagger}/G]/BG, \beta) \ar[d]^{\bar{\pi}}\\
 B \ar[r]^{f} &BG.}
\end{equation}
Here $\mbar_{0,n}(E^{\dagger}/B,\beta)$ is the log stack of stable log maps to $E^{\dagger}$ over the category of log schemes over $B$ with the prescribed discrete invariants. It can also be viewed as the log stack of stable maps to the family $E^{\dagger} \to B$ over the category of log schemes. $\mbar_{0,n}( [F^{\dagger}/G]/BG, \beta) $ is defined similarly.

\begin{prop}
\label{p:2.0.4}
\[
 [\mbar_{0,n}(E^{\dagger}/B, \beta)]^{\vir}  
= \bar{g}^*[\mbar_{0,n}( [F^{\dagger}/G]/BG, \beta)]^{\vir}.
\]
\end{prop}

\begin{proof}
The cartesian diagram
\[
 \xymatrix{
 \mbar_{0,n}(E^{\dagger}/B, \beta)\ar[r]\ar[d] 
    &\mbar_{0,n}( [F^{\dagger}/G]/BG, \beta) \ar[d]\\
 B \times \mathfrak{M}_{0,n} \ar[r] &BG \times \mathfrak{M}_{0,n} }
\]
induces another
\[
\xymatrix{
 \mbar_{0,n}(E^{\dagger}/B, \beta)\ar[r]\ar[d] 
     &\mbar_{0,n}( [F^{\dagger}/G]/BG, \beta) \ar[d]\\
 B \times \Tor_{\mathfrak{M}_{0,n}} \ar[r] &BG \times \Tor_{\mathfrak{M}_{0,n}},
}
\]
where $\mathfrak{M}_{0,n}$ is the moduli stack of genus zero, $n$-pointed prestable curves and $\Tor_\mathfrak{M}$ is the classifying stack of fine and saturated log schemes over $\mathfrak{M}$ \cite[Remark 5.26]{O1}. Since $\Omega_{E^{\dagger}/B} \simeq g^* \Omega_{\pi}$, the perfect obstruction theories for the vertical arrows are compatible.
By \cite[Theorem 4.1 (3)]{vpb}, virtual pullbacks commute with flat pullbacks, the proposition is proved.
\end{proof}

When $B$ is a point,  from the diagram
\[
\xymatrix{
\mbar_{0,n}(F^{\dagger}, \beta) \ar[r]\ar[d]& \mbar_{0,n}( [F^{\dagger}/G]/BG, \beta) \ar[d]\\
\on{Spec}\cc \ar[r]& BG
}
\]
we see
\begin{cor}
\label{c:2.0.5}
\[
\mbar_{0,n}( [F^{\dagger}/G]/BG, \beta) = [\mbar_{0,n}( F^{\dagger}, \beta)/G ],
\]
and $[\mbar_{0,n}( [F^{\dagger}/G]/BG, \beta)]^{\vir}$ corresponds to the $G$-equivariant virtual class on $[\mbar_{0,n}( F^{\dagger}, \beta)/G ]$.
\end{cor}

We note that $\mbar_{0,n}(E^{\dagger}/B, \beta)$ can be identified with $\mbar_{0,n}(E^{\dagger},\beta)$ as moduli stacks by forgetting the map to the base $B$, and their virtual classes are the same since we are considering curves of genus zero (cf. \cite[Equation (3)]{MP}).

We spell out the consequence of Proposition \ref{p:2.0.4} in numerical form.

\begin{cor}[{cf.~\cite[Equation (4))]{MP}}]
Let $A^*$ be the operational Chow ring.
For any non-negative integers $k_i$, and classes $\delta_i \in A^*(B)$,  
$\alpha_r \in A^*([F/G])$, $\theta_s \in A^*([D_F/G])$, 
we have
\begin{equation} \label{e:2.1.2}
\begin{split}
&
\Big\langle~
\prod_{1\le  r \le n}\taub_{k_r}(p^*(\delta_r)\cup g^*(\alpha_r)) 
\mid \prod_{1 \le s \le m} \taub_{k_{s+n}}(p^*(\delta_{s+n})\cup g^*(\theta_s))
~\Big\rangle^{(E, D_E)}_\beta
\\& =
\int_B \prod_{1 \leq i \leq n+m} \delta_i \cup
f^*\Big\langle~
\prod_{1\le  r \le n}\taub_{k_r}\alpha_r 
\mid \prod_{1 \le s \le m} \taub_{k_{s+n}}\theta_s
~\Big\rangle^{([F/G], [D_F/G])}_\beta,
\end{split}
\end{equation}
where
\[
\begin{split}
&\Big\langle~
\prod_{1\le  r \le n}\taub_{k_r}\alpha_r 
\mid \prod_{1 \le s \le m} \taub_{k_{s+n}}\theta_s
~\Big\rangle^{([F/G], [D_F/G])} \\
&= \bar{\pi}_*
\Big(
\llangle~
\prod_{1\le  r \le n}\taub_{k_r}\alpha_r 
\mid \prod_{1 \le s \le m} \taub_{k_{s+n}}\theta_s
~\rrangle
\Big)
\end{split}
\] 
is an equivariant GW invariant of $(F, D_F)$ with the cycle
$\llangle~ \cdot | \cdot
~\rrangle$ defined in \eqref{rel-cycle} and
$\bar{\pi}$ defined in \eqref{e:2.1.1}.

In particular, the LHS of \eqref{e:2.1.2} 
 is nonzero only if $\sum \deg{\delta_i} \le \dim B$.
\end{cor}

We will not make explicit use of this corollary. 
The interested readers may consult \cite{MP} for details.

\begin{prop} \label{prop:bir}
Given a birational map $x: B' \to B$ between smooth projective varieties, let ${E'}^{\dagger}$ be the fiber product in the following diagram
\[
\xymatrix{
{E'}^{\dagger} \ar[r]^{y}\ar[d] & E^{\dagger}\ar[d]\\
B' \ar[r]^{x} & B,
}
\]
and
\[
\xymatrix{
 \mbar_{0,n}({E'}^{\dagger}, \beta) \ar[r]^{\bar{y}}  
 &\mbar_{0,n}( E^{\dagger}, \beta) }
\]
be the induced map. Then for any fiber curve $\beta$, we have
\[
 \bar{y}_*([\mbar_{0,n}({E'}^{\dagger}, \beta)]^{\vir}
 =[\mbar_{0,n}( E^{\dagger}, \beta)]^{\vir}.
\]
\end{prop}

\begin{proof}
Consider the cartesian diagram
\[
\xymatrix{
 \mbar_{0,n}({E'}^{\dagger}/B', \beta) \ar[r] \ar[d]
 &\mbar_{0,n}( E^{\dagger}/B, \beta) \ar[d]\\
 B' \times \Tor_{\mathfrak{M}_{0,n}} \ar[r] & B \times \Tor_{\mathfrak{M}_{0,n}}}
\]
induced from 
\[
\xymatrix{
\mbar_{0,n}({E'}^{\dagger}/B', \beta)\ar[r]\ar[d] 
&\mbar_{0,n}( E^{\dagger}/B, \beta) \ar[d]\\
B' \times \mathfrak{M}_{0,n} \ar[r] &B \times \mathfrak{M}_{0,n}.
}
\]
Since 
$$
y^*\Omega_{E^{\dagger}/B} \simeq \Omega_{{E'}^{\dagger}/B},
$$ 
the perfect obstruction theories for the vertical arrows are compatible, and the bottom arrow is projective of degree 1. Therefore
\[
\bar{y}_*([\mbar_{0,n}({E'}^{\dagger}, \beta)]^{\vir}
 =[\mbar_{0,n}( E^{\dagger}, \beta)]^{\vir}
\]
by \cite[Theorem 4.1 (3)]{vpb}.
\end{proof}

\section{Invariants of a $\bp^1$ bundle in terms of those of the base}
\label{s:3}

We discuss the relation of the genus zero relative and rubber invariants of a $\bp^1$ bundle with those of its base. The main tools used here are C.~Manolache's virtual pullback and pushforward \cite{vpb, vpf}.

In \cite[Section 5.4]{vpb}, the absolute invariants of a $\bp^1$ bundle and its base are related by the strong virtual pushforward property as in \cite{vpf}. We adapt the arguments there and establish similar results for relative and rubber invariants. For rubber invariants, the rubber calculus in \cite{MP} is also used.

\subsection{Terminologies and notations} \label{s:3.1}

Let $X$ be a smooth projective variety and $L$ a line bundle on $X$. Let $Y=\bp_X(L \oplus \oo)$, which has a natural projection 
$$
\pi:Y \to X.
$$ 
$\pi$ has two sections $Y_0, Y_\infty$. Denote by 
$$
i_0: Y_0 \to Y \quad \text{and} \quad i_{\infty}: Y_{\infty} \to Y
$$ 
the inclusions. 

Recall some terminologies used in \cite{MP}. Relative invariants coming from  $(Y, Y_0)$ and $(Y,Y_{\infty})$ are called \emph{type I}; those from $(Y,Y_0,Y_{\infty})$ are called \emph{type II}.

A variant of type II relative invariants are the invariants of the non-rigid targets, called \emph{rubber targets}, whose relative invariants are called \emph{rubber invariants}.
See \cite[Section~2.4]{GV} and \cite[Section~1.5]{MP} for precise definitions and references.
The rubbers naturally occur in the expanded targets of the usual relative maps. 

A cohomology class of the form ${i_0}_*(\omega)$ or ${i_\infty}_*(\omega)$ is called \emph{distinguished}. Note that 
\[
  [Y_0]\cdot \pi^*\alpha = {i_0}_*(\alpha), \qquad 
  [Y_{\infty}] \cdot \pi^*\alpha = {i_{\infty}}_*(\alpha).
\]
In this section, we will use $\omega$ to denote a \emph{non-distinguished} class, i.e.\ 
$$\omega \in \pi^* H^*(X) \subset H^*(Y).$$

\subsection{Relative invariants with rigid targets}

\subsubsection{Log notations} We have 
$$
A_1(Y)= {i_0}_*A_1(X) \oplus \mathbb{Z}[\bp^1],
$$ 
where $[\bp^1]$ is the class for the fiber of $\pi$. For an effective curve class $\beta$ of $Y$, it is determined by $\theta=\pi_*(\beta)$ and $\int_\beta Y_{\infty}$ by 
\[ 
\beta= {i_0}_*(\theta) + \Big(\int_\beta Y_\infty \Big)[\bp^1]. 
\]

We use $(Y, Y_0, Y_\infty)$ to denote the log scheme whose underlying scheme is $Y$ equipped with the  divisorial log structure determined by the divisor $Y_0\sqcup Y_\infty$. 
Locally it is the product of $U$ (a Zariski open subset of $Y$) with the trial log structure and the log scheme $(\bp^1, \{0\}, \{\infty\})$. Similarly we have a log scheme $(Y, Y_\infty)$. 
They are log smooth and integral.

Let  $\mbar_{0,n}((Y, Y_0, Y_\infty), \beta; \mu, \nu)$ be the log stack of  stable log maps from genus zero , $n$-pointed log curves to $(Y, Y_0, Y_\infty)$. Here $\beta$ is the curve class, $\mu$ a partition of $d_0=\int_\beta Y_0$ and $\nu$ a partition of $d_\infty=\int_\beta  Y_\infty$. This is equivalent to specifying the contact orders of the marked points with $Y_0$ and $Y_\infty$ (see \cite[Section 3.2]{AC} and \cite{GS}). As $\mu$ and $\nu$ encode the log structure we are considering on $Y$ and $\nu$ determines $d_\infty$, we will use the notation $\mbar_{0,n}(Y;\mu, \nu)$ when $\theta=\pi_*(\beta)$ is clear from the context.

For  relative invariants of $(Y,Y_\infty)$ with class 
$$
\beta = {i_0}_*(\theta) + d_\infty[\bp^1] 
$$
and a partition $\nu$ of $d_\infty$, we have the log stack 
$$
\mbar_{0,n}((Y,Y_\infty),\beta ;\nu)
$$ 
or $\mbar_{0,n}(Y;\nu)$ for short.

\subsubsection{A virtual dimension count}

View $X$ as a log scheme with the trivial log structure. The projections
$$
(Y, Y_0, Y_\infty) \to X \quad \text{and} \quad (Y, Y_\infty) \to X
$$ 
are log maps. When $\theta$ is nonzero or $n\ge 3$, we have induced maps between log stacks:
\begin{equation} \label{e:pq}
\begin{split}
 &p^X: \mbar_{0,n}( Y; \mu, \nu) \to \mbar_{0,n}(X, \theta), \\
 &q^X:\mbar_{0,n}( Y; \nu) \to \mbar_{0,n}(X, \theta).
\end{split}
\end{equation}
The following lemma follows from virtual dimension count.

\begin{lemma} \label{dim-count}
\mbox{}
\begin{enumerate}
\item $\dim\, [\mbar_{g, n}( Y; \mu, \nu)]^{\vir} = \dim\,[\mbar_{g, n}(X, \theta)]^{\vir} + 1 - g$.

\item $\dim\, [\mbar_{g, n}(Y;\nu)]^{\vir} =  \dim\,[\mbar_{g, n}(X, \theta)]^{\vir} + 1 - g +\int_\beta Y_0.$
\end{enumerate}
\end{lemma}

When $g = 0$, the lemma suggests we might prove strong virtual pushforward properties for $p^X$ and, when $1+ \int_\beta Y_0\ge 0$, for $q^X$.

\subsubsection{Compatibility of obstruction theories}
Let $X$ and  $X'$ be log smooth projective varieties. Consider a strict map 
$$
i: X \to X';
$$
assume the underlying map of $i$ is either a closed immersion or induces an injective map on the Chow group $A_1$ as in \cite[Section 5]{vpb}.

The map $i$ induces a map between log stacks 
$$
\bar{i}:\mbar_{0,n}(X) \to \mbar_{0,n}(X'),
$$ 
where $\mbar_{0,n}(X)$ (resp.\ $\mbar_{0,n}(X')$) is the log stack of stable log maps to $X$ (resp.\ $X'$) from genus zero, $n$-pointed log curves . (We do not specify curve class or contact orders for ease of notation.) By our assumption on $i$, there is a commutative diagram
\[
\xymatrix@!R=.5pc{
\mbar_{0,n}(X) \ar[rr]^-{\bar{i}} \ar[dr]_{\ft}&&\mbar_{0,n}(X')\ar[dl]^{\ft'}\\
&\mathfrak{M}_{0,n}&
}\]
and the horizontal arrow is strict. This induces a commutative diagram between stacks
\[ 
\label{comm} 
\xymatrix@!R=.5pc{ 
\mbar_{0,n}(X) \ar[rr]^-{\bar{i}}\ar[dr] &&\mbar_{0,n}(X')\ar[dl]\\
&\Tor_{\mathfrak{M}_{0,n}}
}.
\] 
Define
$$
E_{\bar{i}}^{\vee} := \mathbb{R}\pi_*f^*(\bL_i^{\vee}),
$$ 
where $\pi, f$ are maps  from the universal curve $\mathcal{C}$ over $\mbar_{0,n}(X)$ 
in the following diagram
\[
\xymatrix{
\mathcal{C} \ar[r]^{f}\ar[d]^{\pi} & X\\
\mbar_{g,n}(X).&
}
\]
Then it is straightforward to check we have compatible obstruction theories 
\begin{equation} \label{e:3.2.1}
\xymatrix{
\bar{i}^* E' \ar[r] \ar[d] & E \ar[r]\ar[d] &  E_{\bar{i}}\ar[d]\\
\bar{i}^*\bL_{\ft'}   \ar[r]  &   \bL_{\ft}\ar[r]& \bL_{\bar{i}} \ ,
}
\end{equation}
where 
$$
E \to \bL_{\ft'} \quad \text{and} \quad E' \to \bL_{\ft}
$$
are the perfect obstruction theories for $\ft$ and $\ft'$ respectively.

The bottom row of \eqref{e:3.2.1} can be identified with the transitivity triangle of 
Olsson's log cotangent complexes, while the top row is related to the  transitivity triangle on $X$
$$\bar{i}^*\Omega_{X'} \to \Omega_X \to \bL_i .$$

\subsubsection{Strong virtual pushforward property}
The proof of the following proposition is modeled on \cite[Proposition 5.22 and Corollary 5.27]{vpb},
where the absolute invariants are treated.

\begin{prop} \label{prop:svpf}
Let $p^X, q^X$ be morphisms defined in \eqref{e:pq}. We have
\begin{enumerate}
\item
\[
(p^X)_*[\mbar_{0,n}(Y;\mu, \nu)]^{\vir} = 0 \quad\text{in} \quad A_*(\mbar_{0,n}(X, \theta))\]
and 
\[
(p^X)_*([\mbar_{0,n}(Y;\mu, \nu)]^{\vir} \cap ev_1^*[Y_0]) = N(\mu,\nu) [\mbar_{0,n}(X, \theta)]^{\vir},
\]
where $N(\mu,\nu)$ is a rational number determined by $\mu$ and $\nu$. 

\item
Assume $\int_\beta Y_0 \ge 0$, then
\[
(q^X)_*[\mbar_{0,n}(Y;\nu)]^{\vir}=0 .
\]
\end{enumerate}
\end{prop}

\begin{proof}
We will prove the strong virtual pushforward property \cite[Definition~4.1]{vpf} for $p^X$ and $q^X$,
which consists of mainly checking certain compatibility of perfect obstruction theories.
Then the above equations follow from the virtual dimension counts in Lemma~\ref{dim-count}.

For (1), we embed $X$ into a homogeneous variety. 
Choose two line bundles $M$ and $L$ on $X$ such that both $M$ and $L\otimes M$ are very ample. 
These line bundles induce an embedding
\[
i:X \to \bp^{|M|} \times \bp^{|L \otimes M|}
\] 
such that $L$ is the pullback of $\oo(-1,1)$ on $ \bp^{|M|} \times \bp^{|L \otimes M|}$.
Then we have a cartesian diagram
\[
\xymatrix{
Y \ar[r]^-j\ar[d] & \bp( \oo(-1,1) \oplus \oo )\ar[d]\\
X \ar[r]^-i & \bp^{|M|} \times \bp^{|L \otimes M|},
}
\]
which induces a cartesian diagram between log stacks
\[
\xymatrix{
\mbar_{0,n}\left( Y;\mu, \nu \right) \ar[r]^-{\bar{j}}\ar[d]^{p_X}&
\mbar_{0,n}\left ( \bp(\oo(-1,1) \oplus \oo);\mu, \nu\right) \ar[d]^{p}\\
\mbar_{0,n}\left(X, \theta \right) \ar[r]^-{\bar{i}}& 
\mbar_{0,n} ( \bp^{|M|} \times \bp^{|L \otimes M|}, { (\int_\theta  M  ,\int_\theta  L\otimes M) }).
}
\]
Here we use $p$ for $p^{\bp^{|M|} \times \bp^{|L \otimes M|}}$ and $\bar{i}, \bar{j}$ to denote the horizontal maps. 
As $\bar{i}$ is strict, the underlying diagram between stacks is also cartesian.

Recall we have defined obstruction theory $E_{\bar{i}}$ (resp.\ $E_{\bar{j}}$) for $\bar{i}$ (resp.\ $\bar{j}$) in \eqref{e:3.2.1}. 
They fit in the following diagram 
\[
\xymatrix{
p_X^*E_{\bar{i}}\ar[r]^-\approx\ar[d]& E_{\bar{j}}\ar[d]\\
p_X^*\bL_{\bar{i}} \ar[r]  &\bL_{\bar{j}} ~.
}
\] 
(cf.~\cite[Propsition 5.4 (ii)]{vpb}.)

As $\bp^{|M|} \times \bp^{|L \otimes M|}$ is homogeneous,  $E_{\bar{i}}$ is perfect in $[-1,0]$. This implies there exists a virtual pullback $\bar{i}^!$ such that
\[
\bar{i}^![\mbar_{0,n} ( \bp^{|M|} \times \bp^{|L \otimes M|}, \mbox{$(\int_\theta  M  ,\int_\theta  L\otimes M)$})]^{\vir}
=[\mbar_{0,n}\left(X, \theta \right)  ]^{\vir}
\]
and 
\[
\bar{i}^! [
\mbar_{0,n}\left ( \bp(\oo(-1,1) \oplus \oo);\mu, \nu\right)]^{\vir}
=[\mbar_{0,n}\left( Y;\mu, \nu \right) ]^{\vir}.
\]
Note that by Lemma \ref{dim-count}, $p$ satisfies the strong virtual pushforward property since $\bp^{|M|} \times \bp^{|L \otimes M|}$ is homogeneous. Then we can transfer this property to $p^X$ using $\bar{i}^!$.

To determine the number $N$, consider a point
\[l:\bp^1 \to \bp^{|M|} \times \bp^{|L \otimes M|}\] 
in
\[\mbar_{0,n}( \bp^{|M|} \times \bp^{|L \otimes M|}, {(\int_{\pi_*{\beta} }M  ,\int_{\pi_*{\beta} }L\otimes M)} ).\]
As 
\[A_1(\bp^1) \to A_1(\bp^{|M|} \times \bp^{|L \otimes M|})\] is injective,
we have a cartesian diagram \small
\[
\xymatrix{
\mbar_{0,n}\left( \bp(\oo(d_0-d_\infty) \oplus \oo);\mu, \nu \right) \ar[r]\ar[d]^{p_{\bp^1}}&
\mbar_{0,n}\left ( \bp(\oo(-1,1) \oplus \oo);\mu, \nu\right) \ar[d]^{p}\\
\mbar_{0,n}(\bp^1, 1) \ar[r]^-{\bar{l}}& 
\mbar_{0,n}( \bp^{|M|} \times \bp^{|L \otimes M|}, {(\int_\theta M  ,\int_\theta L\otimes M)}).
}
\] \normalsize
Note that $ \int_\theta L =d_0 - d_\infty$.

Let 
$$\bp=\bp(\oo(d_0-d_\infty) \oplus \oo). $$
The number $N$ is determined by
\[
(p^{\bp^1})_*\left( [\mbar_{0,n}( \bp ;\mu, \nu)]^{\vir}\cap ev_1^*[\bp_0]\right) 
= N(\mu,\nu)[\mbar_{0,n}(\bp^1, 1)]^{vir}.
\]
This completes the proof of (1).
The proof of (2) is entirely similar and is omitted.
\end{proof}

\begin{remark}
Using localization, one can show similar vanishing results. Consider the fiberwise $C^*$ action on $Y$ and the trivial action on $X$. Under these actions,  $\pi:Y \to X $ is $\cc^*$ equivariant, and the induced map $p^X$ and $q^X$ are $\cc^*$ equivariant. Assume $\beta$ satisfies $\int_\beta Y_0\ge 0,  \int_\beta Y_\infty \ge 0$.

For type II invariants, assume $2g-2+n>0$, then the pushforward of the equivariant virtual class under $p^X$ lies in $t^{-1} A_*(\mbar_{g,n}(X, \gamma))[\![t^{-1}]\!]$ by dimension count. Since pushforward of equivariant class should be an equivariant class, i.e.\ in $A_*(\mbar_{g,n}(X, \gamma))[\![t ]\!]$, it must vanish.

For type I, when $g=0, n\ge 3$, the pushforward of the equivariant virtual class under  $q^X$ lies in $t^{-1} A_*(\mbar_{0,n}(X, \gamma))[\![t^{-1}]\!]$ by dimension count and therefore vanishes.
\end{remark}

\subsubsection{Special two-pointed fiber integrals}

When $\theta$ is zero, $\beta$ is a fiber class for $\pi$
and $d_0=d_\infty >0$ for type II invariants.
In particular, 
\[
 n \ge l(\mu)+l(\nu) \ge 1+1= 2.
\]
If $n=2$, we see $l(\mu) =l(\nu)=1$ 
and there are no non-relative marked points.
Let $d := d_0 =d_\infty$.

\begin{lemma} \label{lem:rootst}
When the partitions $\mu = (d), \nu =(d)$ are totally ramified
and $\beta$ a fiber class,
$\mbar_{0,2}( Y; (d), (d))$ is isomorphic to the root stack $\sqrt[d]{L/X}$ 
(\cite[Appendix B.1]{AGV}).
In particular, it is smooth with virtual class equal to its fundamental class.
Consequently, for $\alpha_1, \alpha_2 \in H^*(X)$,
\[
\int_{[\mbar_{0,2}( Y; (d), (d))]^{\vir}} \ev_1^*(\alpha_1) \cup \ev_2^*(\alpha_2) = \frac{1}{d}\int_X \alpha_1 \cup \alpha_2.
\]
\end{lemma}
 
\begin{proof}
We show that first that the source curve has no contracted component. 
Assume there are $v$ contracted components and $h$ non contracted components, then there are $v+h-1$ nodes.
Consider the number of  special points (nodes or marked points) on each component.  On a contracted component,
there are at least 3 of them. There are at least 2 special points on a non contracted component, which are points mapped into $Y_0$ and $Y_\infty$.
As each node is counted twice,  we have
\[
 2(v+h-1)+2 \ge 3v + 2h.
\]
Thus $v=0$.
This implies in fact the source curve must be smooth as this is a 
fiber integral with 2 totally ramified relative points.

It is then easy to see a stable log map $C \to Y$
is determined by its underlying map. 
The moduli space being unobstructed follows from 
$$H^1(C, \mathscr{O}_C)=0.$$

Since $\beta$ is a fiber class for $\pi$, the last integral can be
evaluated by first integrating over the fiber of
\[
 e: \mbar_{0,2}( Y; (d), (d)) = \sqrt[r]{L/X}  \to X
\]
which has degree $1/d$.
The last statement then follows.
\end{proof}





\subsection{Rubber invariants}

Let $\mbar_{\Gamma}(Y, \beta; \mu, \nu)^{\bullet \sim}$ be the moduli stack of relative maps to rubber targets. Here $\Gamma$ specifies the discrete data for the genus zero source curve, including the number of components, the curve class of each component, and the distribution of marked points among these components. 

We call a component \emph{unstable} if its curve class is a fiber class for $\pi$, and there are only two relative marked points with no other marked points. 
Otherwise the component is \emph{stable}.
Unstable component might appear for disconnected rubber invariants.

We treat the rubber invariants in two steps.

\begin{lemma} 
Suppose there is no unstable component in $\Gamma$. 
In this case $\pi: Y \to X $ induces 
\[
 r^X: \mbar_{\Gamma}(Y, \beta; \mu, \nu)^{\bullet \sim} \to 
 \mbar_{\Gamma}(X, \theta)^{\bullet}.
\]
Furthermore, under the same assumption, we have
 \[
 \dim [\mbar_{\Gamma}( Y,\beta; \mu, \nu)^{\bullet \sim}]^{ \vir} = 
 \dim[\mbar_{\Gamma}(X, \theta)^{\bullet}]^{\vir}+ c(\Gamma)-1,
\]
where $c(\Gamma)$ is the number of components (of domain curve) in $\Gamma$.
\end{lemma}

\begin{proof}
The first statement is an easy consequence of the definitions of the moduli 
and the second follows from a straightforward dimensional count.
\end{proof}

\begin{prop} \label{prop:svpf'}
When $\Gamma$ does not contain an unstable component, we have
\[
(r^X)_*([\mbar_{\Gamma}(Y,\beta ;\mu, \nu)^{\bullet \sim} ]^{\vir}=
\left \{
\begin{array}{ll}
 R(\mu,\nu) [\mbar_{\Gamma}(X, \theta)^{\bullet}]^{\vir} ,
  &\text{if}  ~c(\Gamma)=1.\\
0 ,& \text{ otherwise}.
\end{array} 
\right.
\]
where $R(\mu,\nu)$ is a rational number determined by $\mu, \nu$.
\end{prop}

\begin{proof}
The proof is similar to the one given in Proposition~\ref{prop:svpf}.
\end{proof}

\begin{remark}
One may use a variant of Bumsig Kim's log moduli stack, adapted to rubber targets in \cite{MR}, 
to check compatibilities of perfect obstruction theories in Proposition~\ref{prop:svpf'}.
The log moduli stack of Kim is the saturation of the log moduli stack of Jun Li. 
See \cite[Section 6]{GS}.
\end{remark}

In general, insertions $\omega$ for genus zero rubber invariants
\[
\langle~ 
\vec{\mu} \mid \omega \mid \vec{\nu}~
\rangle^{\bullet Y \sim}_\Gamma
\]
are necessarily non-distinguished.
$\Gamma$ might contain unstable components.
We will show that it can be calculated as a product of contribution from the unstable components and the contribution from the stable component.
The unstable contribution will be treated using Lemma~\ref{lem:rootst} and the stable contribution will be converted to type II invariant with rigid $Y$.

We decompose $\Gamma$ into stable and unstable parts
\[
 \Gamma = \Gamma_s \sqcup \Gamma_u,
\]
where $\Gamma_s$ and $\Gamma_u$) consist of stable and unstable components respectively.
We then divide weighted partitions and insertions accordingly:
\[
 \vec{\mu} =\vec{\mu}_u \sqcup \vec{\mu}_s, \quad \omega= \omega_u \sqcup \omega_s,
 \quad \vec{\nu} =\vec{\nu}_u \sqcup \vec{\nu}_s.
\]

\begin{prop} \label{prop:rubbercalc}
\begin{equation} \label{e:3.2.3}
\langle~ 
\vec{\mu} \mid \omega \mid \vec{\nu}~
\rangle^{\bullet Y \sim}_\Gamma = 
\langle~ 
\vec{\mu}_s \mid \omega_s \mid \vec{\nu}_s~
\rangle^{\bullet Y \sim}_{\Gamma_s} \cdot
\langle~ 
\vec{\mu_u} \mid \omega_u \mid \vec{\nu}_u~
\rangle^{\bullet Y}_{\Gamma_u} ,
\end{equation}
where $\omega$ is a non-distinguished insertion. Furthermore,
$\langle~ 
\vec{\mu}_s \mid \omega_s \mid \vec{\nu}_s~
\rangle^{\bullet Y \sim}_{\Gamma_s}$
is determined by invariants on $X$ by Proposition~\ref{prop:svpf'}, and 
$\langle~ 
\vec{\mu_u} \mid \omega_u \mid \vec{\nu}_u~
\rangle^{\bullet Y}_{\Gamma_u}$  is determined by Lemma~\ref{lem:rootst}.
In particular, a genus zero rubber invariant is nonzero only if there is exactly one stable component 
in $\Gamma$. 
\end{prop}

\begin{proof}
Note that the stable part $\Gamma_s$ is non-empty by the stability condition.
If there is no internal marked point in $\Gamma_s$, the rubber calculus may be used to create a point as in \cite[Section~1.5]{MP}.
Then we apply the rigidification and replace the rubber invariants by type II invariants (with rigid target $Y$), 
If the extra marked point created above (by rubber calculus) goes to the unstable part, the contribution from that configuration is zero.
Therefore, we may assume the extra point goes to the stable part, which is not empty by the stability requirement.
Once it is rigidified, the product formula applies and we conclude that the contribution splits into the stable contribution and the unstable contribution of the rigid target.
Reverse the rigidification process in the stable contribution, we obtain \eqref{e:3.2.3}.

For the contribution from the stable components, Proposition \ref{prop:svpf} (1) implies that
the stable contribution can be determined by the invariants on $X$ since $\omega$ is non-distinguished (pull-backed from $X$).
It vanishes unless the number of connected components $c(\Gamma)$ is $1$.

The curve classes for unstable components are necessarily fiber curve classes.
After the rigidification, they can be treated by Lemma~\ref{lem:rootst} for the same reason that $\omega$ is a non-distinguished class.
\end{proof}

\section{Deformation invariance of $\T$-invariance property} \label{s:4}

\setcounter{subsection}{+1}

In this section we prove that $\T$-invariance is stable under deformation of vector bundles.
This is an easy consequence of the deformation invariance for GW invariants.

\begin{prop} \label{prop:definv}
Let $E, E'$ be two vector bundles of rank $ r+1$ over $S \times T$, where $S, T$ are smooth projective varieties. If for some $t_0 \in T$, $\T$-invariance holds for $(S, E_{t_0}, E'_{t_0})$, 
then $\T$ invariance holds for $(S, E_{t}, E'_{t})$ any $t \in T$. Here $E_t, E'_t$ are the restrictions of $E, E'$ to $S \times \{ t\}$ for a closed point $t \in T$.


\end{prop}

\begin{proof}

We abuse notation and use $i_t$ to denote both the inclusions  
\[
 S \times \{ t\} \hookrightarrow S \times T \quad \text{and} \quad
 X(S, E_t, E'_t) \hookrightarrow X(S\times T).
\]
Because $T$ is irreducible, it is connected as a complex manifold. Therefore
\[
 {i_t}_* :H_2(S) \to H_2(S \times T)
\] 
is independent of $t$.

Since 
\[
 \xymatrix{ S \times \{ t \}  \ar[r]^{i_{t} }&S\times T  \ar[r]^{~\pr_S}&S} 
\]
is the identity map, the induced map
\[
H_2(S \times \{t\}) \to H_2(S\times T)
\] is injective. 
It follows from the commutative diagram
\[
\xymatrix{
X(S, E_t, E'_t) \ar[r]\ar[d] & X(S\times T)\ar[d]\\
S \times \{t \} \ar[r]  & S \times T\\
}
\]
that $H_2(X(S, E_t, E'_t)) \to H_2(X(S\times T))$ is injective.

Given $\beta \in H_2(X(S \times T))$, consider the cartesian diagram
\[
\xymatrix{
 \mbar_{0,n}(X(S, E_t, E'_t), \beta_t) \ar[r]\ar[d] 
   & \mbar_{0,n}(X(S\times T), \beta)\ar[d]\\
 \{t \} \ar[r] & T .\\
}\]
Here $\beta_t \in H_2(X(S, E_t, E'_t))$, if exists, is the class satisfying ${i_t}_*(\beta_t) =\beta$.

By the compatibility of virtual classes, we have
$$
\Big\langle~~\prod_{i} \pr_S^*(\alpha_i) Q_i
~\Big\rangle^{X(S \times T)}_{\beta}
=\Big\langle~~~\prod_{i}  \alpha_i Q_i ~\Big\rangle^{X(S, E_t, E'_{t})}_{\beta_t}
    \cdot[T] 
$$
as top dimensional cycles in $H_*(T)$. Here $\alpha_i$ are classes on $S$, $Q_i$ are classes on 
$[X(\cc)/G]$.


Now we see the generating function
\[
\Big\langle~\prod_{i}  \alpha_i Q_i ~\Big\rangle^{X(S, E_t, E'_{t})}_{(\beta_S,\,d)}
\]
is determined by 
\[
\Big\langle~\prod_{i} \pr_S^*(\alpha_i) Q_i
~\Big\rangle^{X(S \times T)}_{((i_t)_*(\beta_S), \,d)}
=
\Big\langle~\prod_{i}  \alpha_i Q_i ~\Big\rangle^{X(S, E_t, E'_{t})}_{(\beta_S, \, d)}
\cdot[T]~.
\]
Note that the LHS is independent of $t$.

Similarly, using the fact that $\T$ commutes with pullback, we see that
\[
\Big\langle~\prod_{i} \T(\pr_S^*(\alpha_i)Q_i)
~\Big\rangle^{X'(S \times T)}_{((i_t)_*(\beta_S), \,d)}
=
\Big\langle~\prod_{i}  \T(\alpha_i Q_i) ~\Big\rangle^{X'(S, E_t, E'_{t})}_{(\beta_S,\,d)}
\cdot[T].
\]
As the classes $\left( (i_t)_*(\beta_S),d)\right)$ is independent of $T$, we  see that $\T$-invariance for $(S, E_t, E'_t)$ is independent of $t$.
\end{proof}



\section{$\T$-invariance for $\bp^1$ bundles: absolute, type II and rubbers} 

In this and the next sections, we show that $\T$-invariance for $(D, F, F')$ implies $\T$-invariance for $(P, \pi^*F, \pi^*F')$ for
\[
 \pi: P=\bp_D(N\oplus \oo) \to D
\]
a \emph{split} $\bp^1$ bundle, in the absolute and relative settings. 
In this section, we treat the absolute and type II invariants, including the rubber invariants.

\subsection{Absolute invariants of split projective bundles} 
For absolute invariants, a variant (and easier version) of the quantum Leray--Hirsch theorem proved in Part II \cite{LLW2} 
leads to a stronger result for ordinary flops with base being a split projective bundle. 

\begin{thm} \label{QLH-fiber}
Let 
$$V = \bigoplus_{i = 1}^m L_i \to T$$
such that $L_i$'s are line bundles and 
$$\pi: P = \bp_T(V) \to T$$ 
be its associated split projective bundle. 
Then, $\T$-invariance for generating function of absolute invariants on $(T, F, F')$ 
implies that on $(P, \pi^*F, \pi^*F')$.
\end{thm}

\begin{proof}
This follows from the techniques in \cite{LLW1,LLW2}. Notice that
\[
 X(P) = \mathbb{P}_{X(T)}(V) \to X(T),
\]
is a split $\bp^{m - 1}$ bundle with $V$ pulled-back from the base $T$. The $I$-function of $X(P)$ is a hypergeometric modification of the $J$-function of $X(T)$, as explained in \cite[\S 2]{LLW2}. Symbolically, we can write
\[
 I_{X(P)} = I_{X(P)/X(T)} * J_{X(T)}
\]
where the hypergeometric factor $I_{X(P)/X(T)}$ is determined by the Chern classes of the line bundles $L_i$'s. Since these bundles are \emph{pulled-backs} from $T$,  the 
$\T$-invariance holds by assumption.

Using the techniques in \cite[\S2 and 3]{LLW2}, in particular the $\T$-invariance of the Birkhoff factorization procedure, we conclude that $\T$-invariance for $T$ implies that for $P$.
\end{proof}

Applying the theorem to $T = D$, $V = N \oplus \mathscr{O}$, we obtain the desired result.

\begin{cor} \label{QLH-fiber-2}
$\T$-invariance for generating function of absolute invariants on $(D, F, F')$ 
implies that on $(P, \pi^*F, \pi^*F')$.
\end{cor}

We now proceed to establish similar results for relative and rubber invariants.

\subsection{Type II} \label{s:5.2}
\subsubsection{Fiber classes} \label{s:5.2.2}
We \emph{define fiber curve classes (modulo extremal rays)} to be curve classes 
$(\beta_P, d)$ such that $\beta_P$ is a fiber class for $P \to D$.
Equivalently, $(\beta_P, d)$ is a fiber curve class in this sense if
any $\beta \in (\beta_P,d)$ is a fiber class 
for the fiber bundle $X(P) \to D$ in the sense of section~\ref{s:2}.

\begin{thm} \label{thm:fibertypeii}
Fiber class type II invariants of $(P, P_0,P_\infty)$
are $\T$-invariant.
\end{thm}

\begin{proof}
Let
\[
 G:=GL_{r+1} \times GL_{r+1}, \qquad 
 T :=\prod_1^{2r+2}G_m,
\]
where $G_m = GL_1(\cc)$.

Fiber class invariants for $(P,P_0,P_\infty)$ are fiber integrals.
On the $X(P)$ side, they are fiber integrals for the bundle $X(P) \to D$,
which fits in a cartesian diagram 
\[
\xymatrix{
(X(P),X(P_0),X(P_\infty)) \ar[r]\ar[d] 
   & ([\bp^1 /G_m], [\{0\}/G_m], [\{\infty\}/G_m]) \times [X(\cc)/G] \ar[d]\\
D \ar[r] & BG_m \times BG.
}
\] 

By Proposition \ref{prop:bir}, 
we only need to prove invariance after passing to $D'$ via a birational map $D' \to D$.
Then by Lemmas~\ref{l:5.3.1}, \ref{l:deform}, and Proposition~\ref{prop:definv}, 
we can assume the fiber bundle below has a smaller structure group:
\[
\xymatrix{
(X(P'),X(P'_0),X(P'_\infty)) \ar[r]\ar[d] 
   & ([\bp^1 /G_m], [\{0\}/G_m], [\{\infty\}/G_m]) \times [X(\cc)/T] \ar[d]\\
D' \ar[r] & BG_m \times BT.
}
\] 
Here $P'$ is the pullback of $P$ via $D' \to D$.
Now  fiber class type II invariants of $X(P')$ are determined by (i) the classical intersection product on $D'$, which are the same on both sides of the flop, and (ii) the equivariant invariants of
$$
\big(\bp^1 , \{0\},\{\infty\}\big) \times X(\cc),
$$ 
which by the product formula \cite{LQ} are determined by the equivariant invariants 
of $(\bp^1, \{0\},\{\infty\})$ and $X(\cc)$.
From here we see fiber class $\T$-invariance for $(P,P_0,P_\infty)$ 
follows from that for $T$-equivariant  invariants
of $(\Spec{\cc}, \oo^{\oplus r+1},\oo^{\oplus r+1})$, 
which can be deduced from  the split case for 
\[  \Big\{ \prod_1^{2r+2} \bp^n , ~\boxplus\oo(1), ~\boxplus\oo(1) \Big\}_{n \ge 1 } ,\]
as the Chow group $A_*(BT)$ is 
determined by
$\{ A_*(\prod \bp^n)\}_{n \ge 1}$.
\end{proof}


\subsubsection{Non-fiber classes} \label{s:5.3}

\begin{thm}  \label{thm:typeii}
$\T$-invariance for $D$ implies $\T$-invariance for $(P, P_0,P_\infty)$ 
with curve classes $(\beta_P, d)$ such that $\beta_P$ is not a fiber class 
for $P \to D$.
\end{thm}

\begin{proof}
We prove it by induction on the number of \emph{distinguished insertions} (defined in Section~\ref{s:3.1}). 
If the number of distinguished insertions is less than 2, we apply Proposition~\ref{prop:svpf} (1) to conclude the proof.

For type II invariants of $(X(P), X(P_0),X(P_\infty))$ with $l \ge 2$ distinguished insertions,  we use the degeneration formula. As 
\begin{equation} \label{move}
  [X(P_0)] - [X(P_\infty)] = \pi^*c_1(N|_{X(D)}),
\end{equation}
(pullback of $N$ under $X(D) \to D$), modulo type II invariants with $l-1$ distinguished insertions, we can assume one of the distinguished insertions is of the form ${i_0}_*(\alpha)$ and the other ones are of the form ${i_\infty}_*(\alpha_i)$.

Consider the family 
\[
 W(X(P),X(P_\infty))=\bl_{X(P_\infty) \times \{0\} } (X(P) \times \A^1)
\] 
with divisors $X(P_0) \times \A^1$ and $\widetilde{X(P_\infty) \times \A^1}$, the strict transformation of $X(P_\infty) \times \A^1$ under the blowing-up $W \to X(P) \times \A^1$.

The degeneration formula for this family implies 
\begin{multline*}
\Big\langle~\ \vec{\mu}\mid \omega \cdot  {i_0}_*(\alpha) \prod_{i=1}^{l-1}{i_\infty}_*(\alpha_i ) \mid \vec{\nu}\ 
~\Big\rangle_{(\beta_P, d)}^{(X(P),\,X(P_0),\,X(P_{\infty}))}
= \\ \sum_{I}\sum_{\eta}
C_\eta\,
\langle~\ \vec{\mu}\mid \omega_1 \cdot {i_0}_*(\alpha)\mid \lambda, e^I ~\rangle_{\Gamma_1}
\cdot
{p_X}_*\Big( \Big\langle~ \lambda, e_I \mid \omega_2 \cdot  \prod_{i=1}^{l-1}{i_\infty}_*(\alpha_i ) \mid \vec{\nu} \ ~\Big\rangle_{\Gamma_2} \Big).
\end{multline*}
Note that the RHS is determined by type II generating functions with at most $l-1$ distinguished insertions. 
This relation is clearly compatible with $\T$.
The theorem now follows by induction.
\end{proof}

\subsection{Rubber invariants}

\begin{thm}  \label{thm:rubber}
$\T$-invariance  for $D$ implies $\T$-invariance for (disconnected)
rubber invariants of $(P, P_0,P_\infty)$ with discrete invariants encoded by $\Gamma$, provided $\Gamma$ contains
 a component with at least 3 marked point.
\end{thm}

\begin{proof}
Recall that the insertions for rubber invariants are necessarily non-distinguished.
By Proposition \ref{prop:svpf'}, we see connected rubber invariants with at least 3 marked points are $\T$-invariant. 
Note that a component with at least 3 marked points is stable. 
For disconnected rubber invariants, by Proposition \ref{prop:rubbercalc}, $\T$-invariance follows from that for connected invariants.
\end{proof}

\section{$\T$-invariance for $\bp^1$ bundles: type I invariants}
In this section, we prove $\T$-invariance for type I relative invariants of the $\bp^1$ bundle $P$, assuming invariance of the base $D$.

As the strong virtual pushforward property for type I invariants is conditional, the arguments in this section is not as straightforward as those for type II and rubber in the last section. The starting point is that, up to type II invariants, distinguished insertions in an invariant can be removed using the degeneration formula (Lemma \ref{lem:rmdist}).

Consider a type I invariant of $(P,P_\infty)$ without distinguished insertions of classes $(\beta_P, d)$. If $\int_{\beta_P} P_0 \ge 0$, then the invariant is zero by the strong virtual push forward property (Proposition~\ref{prop:svpf} (2)). 
If $\int_{\beta_P} P_0 < 0$, using an inversion of degeneration argument, we show in Proposition~\ref{prop:inversion'} that the invariant is determined by an absolute invariant, lower order type I invariants, type II and rubber invariants. As absolute, type II and rubber invariants are $\T$-invariant, $\T$-invariance for type I invariants follows.

The proof of Proposition~\ref{prop:inversion'} requires some preparation. 
The argument involves applying the degeneration formula to an \emph{ancestor invariant}. 
We then need to understand certain fiber integral and relative invariants with ancestor insertions. 
A number of lemmas in this section are about non-vanishing fiber integrals, which ensures certain inversion of degeneration arguments.
Ancestor insertions in a relative invariant can be removed using TRR for ancestors (Appendix~A).
We need to keep track of the orders of type I invariants obtained from  removing ancestors. 
To do so, we introduce in Subsection~\ref{ss:dom} a relation between type I invariants called dominance 
and the order of those relative invariants are characterized in Proposition\ref{prop:ancestor}.

\subsection{Nonvanishing conditions} \label{s:6.1}
In this subsection, we discuss some nonvanishing conditions on $X(\cc)$, or $X(\Spec \cc, \mathscr{O}^{\oplus r+1}, \mathscr{O}^{\oplus r+1})$.
These conditions give ``selection rules'' of the cohomology classes in the highest order terms arising from the degeneration arguments. 
See for example the end of the proof of Proposition~{prop:inversion'} and Section~{ss:dom}.


We use $Q$ and $R$ to represent elements in a chosen (homogeneous) basis of $X(\cc)$;
$\deg Q$ and $\deg R$ their degrees with respect to the real grading of $H^*(X(C))$.
A superscript $\vee$ will be used to represent elements in  the dual basis. In particular, when $Q  \ne R$,
\[
\int_{X(\cc)} Q^{\vee} \cup R = 0.
\] 

We use notations form Subsection~\ref{GWI}, especially (\ref{cycle}) and (\ref{rel-cycle}) on certain cycles in absolute and relative moduli stacks. 


\begin{lemma} \label{lem:cycle}
\[
\st_* \Big (\llangle~ Q^{\vee} , R , \underbrace{1,\cdots, 1}_{m} ~\rrangle^{X(\cc)}_{n\ell} \Big)
\]
is a nonzero top dimensional cycle on $\mbar_{0,m+2}$ if and only if $n=0$ and $Q=R$.
\end{lemma}

\begin{proof}
By calculating the dimension of the cycle $\llangle~ Q^{\vee} , R , 1,\cdots, 1 ~\rrangle^{X(\cc)}_{n\ell}$ , we see that  
\[
\deg Q^{\vee}  +\deg R =2r+1.
\]
We can write $Q^{\vee}$ and $R$ uniquely as polynomials in $h$ and $\xi$ of the form
\[
\sum_{i \le r, j \le r+1} a_{ij}h^i\xi^j.
\]
Here $a_{ij}$ are complex numbers. We will write $Q^{\vee}(h,\xi)$ or $R(h,\xi)$ when we view them as polynomials in $h$ and $\xi$.

When $n \ne 0$,  a stable map to $X(\cc)$ factors through $Z(\cc)$. As $\xi|_{Z(\cc)} =0$, 
\[
\llangle~ Q^{\vee}(h,\xi),  R(h,\xi), 1, \cdots, 1~\rrangle^{X(\cc)}_{n\ell} = 
\llangle~ Q^{\vee}(h, 0), R(h,0), 1\cdots 1)~\rrangle^{X(\cc)}_{n\ell},
\]
but  then 
$$
\deg Q^{\vee}(h,0) + \deg R(h,0) \le r+r < 2r+1.
$$

When $n=0$, 
\[
\st_* \Big (\llangle~ Q^{\vee} , R , \underbrace{1,\cdots, 1}_{m} ~\rrangle^{X(\cc)}_{n\ell} \Big) =
\Big(\int_{X(\cc)} Q^{\vee} \cup R\Big) [\mbar_{0,m+2}].\] 
It is nonzero if and only if $Q=R$.
\end{proof}

\begin{lemma}  \label{lem:unstable}
For any $c>0, n \ge 0$ integers, consider the two-pointed relative invariant on $(\bp^1, \, \{ 0\})\times X(\cc)$
\[
\langle~(c, Q^{\vee}) \mid tR~\rangle^{(\bp^1, \, \{ 0\})\times X(\cc)}_{c[\bp^1]+n\ell}
\] 
with curve class $d[\bp^1] + n\ell$, where $[\bp^1]$ the fundamental class of $\bp^1$, and 
$$
t=c_1(\mathscr{O}_{\bp^1}(1)).
$$
Then this invariant is nonzero if and only if $c=1$, $n=0$ and $Q=R$.
\end{lemma}

\begin{proof}
Using the divisor equation, we have
\[
\langle~(c, Q^{\vee} ) \mid tR~\rangle^{(\bp^1, \, \{ 0\})\times X(\cc)}_{c[\bp^1]+n\ell} =
\frac{1}{c}\langle~(c, Q^{\vee} ) \mid t, tR~\rangle^{(\bp^1, \, \{ 0\})\times X(\cc)}_{c[\bp^1]+n\ell}.
\]
By the product formula \cite{LQ},
\[
\langle~(c, Q^{\vee} ) \mid t, tR~\rangle^{(\bp^1, \ \{ 0\})\times X(\cc)}_{c[\bp^1]+n\ell}
=\langle~(c, 1 ) \mid t, t~\rangle^{(\bp^1, ~\{0\})} \langle~ Q^{\vee}, 1,  R~\rangle^{X(\cc)}_{n\ell}.
\]
$\langle~(c, 1 ) \mid t, t~\rangle^{(\bp^1, ~\{0\})}$ is zero unless $c=1$, which follows from calculating the virtual dimension. For $d=1$, $\langle~(1, 1 ) \mid t, t~\rangle^{(\bp^1, ~\{0\})}=1$.

If $\langle~ Q^{\vee}, 1,  R~\rangle^{X(\cc)}_{n\ell}$ is nonzero, then by Lemma~\ref{lem:cycle}$, n=0$ and $Q=R$.
\end{proof}

\subsection{Dominance relation for weighted partitions} \label{ss:dom}
In this subsection, the base $D$ is general. Let $\{ \delta_i \}$ be a basis of $H^*(D)$, we fix a $\T$-homogeneous basis compatible with $\{ \delta_i \}$. We will use $Q$ and $R$ with possible subscripts to represent cohomology classes in $H^*([X(\cc)/G])$ that appear in this
 $\T$-homogeneous basis.
Notations for weighted partitions are introduced in Section~\ref{s:1.5}.

\subsubsection{Dominance relation}
We start with the definition of splitting. They will be used in the definition of dominance relation (Definition~\ref{d:dominance}).

\begin{defn} \label{d:6.2.1}
We say that we can \emph{split} $(\delta_{a_1}, Q)$ into $\{ (\delta_{b_m}, R_m)\}$,
if
\[ \int_D \delta_{a_1}^{\vee} \cdot  \prod_m \delta_{b_m} \ne 0,
\]
and for some $n \ge 0$,
\[
\st_* \Big( \llangle~ Q^{\vee} \cdot \prod_{m} R_m ~\rrangle^{X(\cc)} _{n\ell} \Big)
\]
is a nonzero top dimensional cycle on $\mbar_{0, 1+ | \{ m\} |}$ .
\

In particular, when $(\delta_{a_1}, Q)$ splits into $\{ (\delta_{b_m}, R_m)\}$,
we have
\[
\deg \delta_{a_1} = \sum \deg \delta_{b_m}, \qquad
\deg Q =\sum  \deg R_m.
\]
\end{defn}

\begin{remark}
In the definition of splitting, 
classes $Q$ and $R_m$ are classes  on $[X(\cc)/G]$.
We abuse notation to use the same symbols for 
corresponding classes on $X(\cc)$. 

\end{remark}

\begin{defn}[Dominance relation] \label{d:dominance}
Let
\[ 
 \vec{\mu}=\{(\mu_i,\delta_{a_i} Q_i ) \}_{i \in I} \quad \text{and} \quad 
 \vec{\nu}=\{(\nu_j,\delta_{b_j}R_j)\}_{j \in J}
\]
be two weighted partitions 
such that 
\[ 
 \deg_D{\vec{\mu}} = \deg_{D}\vec{\nu}.
\] 
We say \emph{$\vec{\mu}$ dominates $\vec{\nu}$} 
if there exists a partition of $J$
\[ J = \sqcup_{i \in I} J_i,\]
such that for every $i$, one of the following conditions is satisfied.
\begin{itemize}
\item $J_i \neq \emptyset$ and ${\displaystyle \mu_i -1 >  \sum _{j \in J_i} (\nu_j-1)}$,
\item  $J_i \neq \emptyset$, ${\displaystyle \mu_i -1 =  \sum _{j \in J_i} (\nu_j-1)}$, and $\{ (\delta_{b_j} , R_j )\}_{j \in J_i}$ can be obtained from
  $(\delta_{a_i}, Q_i)$ by splittings.
\item $J_i = \emptyset$ and $\mu_i-1 > 0$,
\item $J_i =\emptyset$, $\mu_i=1$ and $\delta_{a_i}=1$, $Q_i=1$.
\end{itemize}
\end{defn}

The following lemma is an easy consequence of the definition of dominance.

\begin{lemma}
\begin{enumerate}
\item 
If $\vec{\lambda}$ dominates $\vec{\mu}$ and $\vec{\mu}$ dominates $\vec{\nu}$,  then 
$\vec{\lambda}$ dominates $\vec{\nu}$.
\item 
If $\vec{\mu_i}$ dominates $\vec{\nu_i}$, 
then $\sqcup_i ~\vec{\mu_i}$ dominates $\sqcup_i ~\vec{\nu_i}$.
\end{enumerate}
\end{lemma}

\begin{lemma} \label{lem:order}
If $\vec{\mu}$ dominates $\vec{\nu}$ and $\sum \mu_i = \sum \nu_j$, then either 
$\vec{\mu}$ is of lower order than $\vec{\nu}$, 
or it can be identified with $\vec{\nu}$. 
\end{lemma}

\begin{proof}
Assume $\vec{\mu}$ is not of lower order than $\vec{\nu}$.
Let 
\[
 \vec{\mu} = \{(\mu_i, \delta_{a_i}Q_i) \}, \qquad
 \vec{\nu}= \{ (\nu_j, \delta_{b_j}R_j) \}.
\]
By the definition of dominance,
\[ \mu_i -1  \ge \sum_{j \in J_j} (\nu_j - 1),\]
where the right hand side is understood to be zero if $J_i = \emptyset$.
Summing up all $i$, we have 
\[
\sum \mu_i - l(\mu) \ge \sum \nu_j - l(\nu),
\]
or
\[ 
l(\mu) \le l(\nu).
\]
Since $\vec{\mu}$ is not of lower order than $\vec{\nu}$, we should $l(\mu) = l(\nu)$,
or equivalently
\[\mu_i -1  = \sum_{j \in J_j} (\nu_j - 1).\]

Assume $1 \in J_1$ and $\{ (\nu_j , \deg \delta_{b_j} , \deg R_j) \}$ is arranged lexicographically.
We shall compare $(\nu_1, \deg \delta_{b_1}, \deg R_1)$ with $(\mu_1, \deg \delta_{a_1}, \deg Q_1)$.

From the following equations
\[
 \begin{split}
 \mu_1-1 &= \sum_{j \in J_1}(\nu_j -1), \\
 \deg \delta_{a_1} &= \sum_{j \in J_1} \deg \delta_{b_j},\\
 \deg Q_1 &= \sum_{j \in J_1} \deg R_j,
 \end{split}
\]
we conclude that either 
$(\mu_1, \deg \delta_{a_1}, \deg Q_1)$ is of lower order or 
\[
(\mu_1, \deg \delta_{a_1}, \deg Q_1)=(\nu_1, \deg \delta_{b_1}, \deg R_1).
\]
In the latter case 
$$(\nu_j, \deg \delta_{b_j}, \deg R_j) = (1,0,0)$$ 
for all $j \in J_1, j \ne 1$.

When 
$$(\mu_1, \deg \delta_{a_1}, \deg Q_1)=(\nu_1, \deg \delta_{b_1}, \deg R_1),$$
$\delta_{a_1}$ can be split into $\delta_{b_1}, 1, \cdots ,1$ and 
$Q_1$ can be split into $R_1, 1, \cdots, 1$. 
It is easy to see $\delta_{a_1} =\delta_{b_1}$ from the definition of a splitting. By Lemma~\ref{lem:cycle}, we see $Q_1=R_1$.
Now we find the highest order term of $\vec{\nu}$
in $\vec{\mu}$. The same argument will match the other terms of $\vec{\nu}$ with terms in $\vec{\mu}$.
\end{proof}

\subsubsection{Generating functions with relative ancestor insertions} \label{s:6.2.2}
Here we consider the only type of generating functions whose relative insertions contain ancestors.
It naturally occur in our induction process.
 
Let $D \subset S$ be a smooth divisor as before.
Consider a special type of relative weighted partition with ancestors
\begin{equation} \label{e:6.2.1}
 \gamma = \{ (\mu_i, \delta_{a_i} Q_i) \} \sqcup  
   \{ (1, \taub_{\nu_j-1}( \delta_{b_j} R_j )\}.
\end{equation}
That is, the relative insertions with ancestors always have contact order (multiplicity) $1$.
We refer to $\gamma$ as specified by two (primary) weighted
partitions 
$$\vec{\mu}= \{ (\mu_i, \delta_{a_i} Q_i) \} \quad \text{and} \quad
\vec{\nu}=\{ (\nu_j, \delta_{b_j} R_j )\}.$$ 
Define 
\[
\deg_D (\gamma) = \deg_D (\vec{\mu}) + \deg_D(\vec{\nu}).
\]

The relative generating functions 
\[
 \langle~ \omega \mid \gamma ~\rangle^{(X(S),X(D))}_{(\beta_S, \,d)}
\]
with $\gamma$ of the above form will be useful in our induction process.
In particular, the multiplicity $1$ condition comes from 
Lemma~\ref{lem:unstable}.

The following proposition shows that the relative generating functions of the above type can be
determined inductively by generating functions whose relative insertions contains no ancestors.

\begin{prop} \label{prop:ancestor}
Let $\langle~ \omega \mid \gamma ~\rangle^{(X(S),X(D))}_{(\beta_S, \,d)}$ be
the relative generating functions, with $\gamma$ specified by \eqref{e:6.2.1}.

This type of relative invariants $\langle~ \omega \mid \gamma ~\rangle^{(X(S),X(D))}_{(\beta_S, \,d)}$ 
can be determined by the following three types of generating functions:
\begin{enumerate}
\item[(i)] primary relative invariants that are of strictly lower order than 
\[\{ (\beta_S,d), |\!|\omega|\!|, \deg_D(\gamma) \},\] 

\item[(ii)] rubber invariants, and 
\item[(iii)] primary relative invariants of order 
$$
\{ (\beta_S,d), |\!|\omega|\!|, \deg_D(\gamma) \},
$$ 
whose weighted partition dominates $\vec{\mu} \sqcup \vec{\nu}$.
\end{enumerate}
\end{prop}

\begin{proof}
If there is no ancestor involved, the statement is trivial by (iii).
Otherwise we apply the topological recursion relation for ancestors and
perform induction on the number of ancestor insertions.

If we apply TRR once to
$\langle~ \omega \mid \gamma ~\rangle ^{(X(S), X(D))}_{(\beta_S, \,d)}$ 
and lower $\taub_{\nu_j-1}$ in its relative insertion to $\taub_{\nu_j-2}$, 
we get a positive linear combination of terms either of the form
(see Appendix A)
\[
\langle~
\omega_A \cdot \alpha \mid \gamma_A
~\rangle_{(\beta_A, \,d_A)}^{(X(S), X(D))}
\cdot
\langle~
\omega_B \cdot \alpha^{\vee} \mid \gamma_B
~\rangle_{(\beta_B, \,d_B)}^{(X(S), X(D))}
\] 
where 
\[
 \begin{split}
 |\!|\omega_A|\!| + |\!|\omega_B|\!| &= |\!|\omega|\!|, \\
 (\beta_A,d_A) + (\beta_B,d_B) &= (\beta_S,d),
 \end{split}
\]
or of the form
\[
\langle~
\omega_{\Xi} \mid \gamma_{\Xi}
~\rangle^{\bullet (X(S), X(D))}_{_{\Xi}}
\cdot {p_X}_*\big(
\langle~
\gamma_{\Xi'} \mid \omega_{\Xi'} \mid \gamma
~\rangle^{\bullet X(P) \sim}_{_{\Xi'}} \big),
\]
where $p_X$ is the composition of maps 
$$X(P) \to X(D) \hookrightarrow X(S).$$

For a term of the form 
\[
\langle~
\omega_A \cdot \alpha \mid \gamma_A
~\rangle_{(\beta_A,\,d_A)}^{(X(S), X(D))}
\cdot
\langle~
\omega_B \cdot \alpha^{\vee} \mid \gamma_B
~\rangle_{(\beta_B,\,d_B)}^{(X(S), X(D))},
\] 
if $(\beta_A, d_A)$ equals $(\beta_S, \,d)$ then there is no relative insertion in $\gamma_B$. As either factor contains at least 3 marked point,  $|\!|\omega_B|\!| +1 \ge 3$. Then 
\[ 
|\!|\omega_A|\!| +1 = |\!|\omega|\!| - |\!|\omega_B|\!| +1 < |\!|\omega|\!|.
\]
Thus either factor gives rise to terms of lower order than $\{ (\beta_S,d), |\!|\omega|\!| \}$.

For  a term of the form
\[
\langle~
\omega_{\Xi}\mid \gamma_{\Xi}
~\rangle^{\bullet (X(S), X(D))}_{_{\Xi}}
\cdot {p_X}_*
\langle~
\gamma_{\Xi'}\mid \omega_{\Xi'}\mid \gamma
~\rangle^{\bullet X(P) \sim}_{_{\Xi'}},
\]
the factor
$
\langle~
\omega_{\Xi} \mid \gamma_{\Xi}
~\rangle^{ \bullet (X(S), X(D))}_{_{\Xi}}
$ will produce primary relative invariants of order $\{ (\beta_S,d) ,|\!|\omega|\!| \}$
only when $\Xi$ specifies a connected curve with classes $(\beta_S,d)$ and $\omega_{\Xi} =\omega$. Then it is of the form 
$\langle~ \omega \mid  \gamma' ~\rangle ^{(X(S), X(D))}_{(\beta_S, \,d)}$ and the curve classes for $\langle~ \gamma_{\Xi'} \mid \omega_{\Xi'} \mid \gamma
~\rangle^{ \bullet (X(P))\sim}_{_{\Xi'}}$ are fiber classes.

If $\langle~
\gamma_{\Xi'} \mid \omega_{\Xi'} \mid \gamma
~\rangle^{ \bullet X(P) \sim}_{_{\Xi'}}$
is nonzero, by Proposition \ref{prop:rubbercalc} we see that
$\Xi'$ contains a stable component with at least 3 marked points, 
and all other possible components are unstable.
For such a configuration, 
the ancestor insertion $\taub_{\nu_j-2}$ belongs to the stable component,
and on any unstable component the contact order and insertion for $\gamma$ and $\gamma'$ should match .
This implies $\gamma'$ is specified by some partitions
$\vec{\mu'}$ and $\vec{\nu'}$, and it 
has at least one less ancestor insertions than $\gamma$.
Therefore we can assume the proposition holds for
$\langle~ \omega \mid\gamma'  ~\rangle ^{(X(S), X(D))}_{(\beta_S, \,d)}$.

As $\langle~
\gamma_{\Xi'} \mid \omega_{\Xi'} \mid \gamma
~\rangle^{ \bullet X(P) \sim}_{_{\Xi'}}~$ is a non-zero fiber integral, we have
\[ \deg_D(\gamma')  \ge \deg_D(\gamma)= \deg_D(\vec{\mu} \sqcup \vec{\nu} ).\]

If $\vec{\mu'} \sqcup \vec{\nu'}$ dominates $\vec{\mu} \sqcup \vec{\nu}$
 when  
$\deg_D(\gamma')  = \deg_D(\gamma)$,
 then by the transitivity
of dominance, the proposition is proved for 
\[\langle~
\omega \mid \gamma
~\rangle^{(X(S),X(D))}_{(\beta_S, \,d)}.\]

The fact that $\vec{\mu'} \sqcup \vec{\nu'}$ dominates $\vec{\mu} \sqcup \vec{\nu}$
follows from a simple dimensional count that the stable rubber component should have
enough points to guarantee the non-vanishing of ancestors invariants.
\end{proof}

\subsubsection{A dominance lemma}
The following lemma will be used in the proof of Proposition~\ref{prop:inversion'}.

\begin{lemma} \label{lem:stable}
If the generating function \[
\Big( \int_{D} 
\delta_{a}^{\vee} \prod_{j} \delta_{b_j} \Big) 
\Big\langle~(c, Q^{\vee}) \mid t^{k} \prod_{j} \taub_{\nu_{j-1}}(tR_j)~
\Big\rangle^{(\bp^1, \,\{0\})\times X(\cc)}_{(c[\bp^1],\,0)} 
\] is nonzero, then
$(c, ~\delta_{a}Q)$ dominates 
$\{ (\nu_j ,  ~\delta_{b_j}R_j)\}$.
\end{lemma}

\begin{proof}

The generating function 
\[
\langle~(c, Q^{\vee}) \mid t^{k} \prod \taub_{\nu_{j-1}}(tR_j)~\rangle^{(\bp^1, \,\{0\})\times X(\cc)}_{(c[\bp^1], 0)} 
\]
is the following sum
\[
\sum_{n \ge0}
\langle~(c, Q^{\vee}) \mid t^{k} \prod  \taub_{\nu_{j-1}}(tR_j)~\rangle_{(c, \,n\ell)} q^{(c,n\ell)},
\]
where
$(c,n\ell)$ represents the curve class $c[\bp^1] +n\ell$ on $\bp^1 \times X(\cc).$

By the product formula \cite{LQ},

$$\langle~(c, Q^{\vee}) \mid t^{k} \prod  \taub_{\nu_{j-1}}(tR_j )~\rangle_{(c,n\ell)}$$
is the intersection  on $\overline{\M}_{0,  1+k+ | \{j \}|}$ of the cycles
$$\st_* \Big( \llangle~~ (c,1) \mid t^{k} \prod \taub_{\nu_{j-1}} (t)~ \rrangle^{ ( \bp^1, \, \{0\}) }_{ c[\bp^1], (c)} \Big)$$ and
 $$\st_* \Big(  \llangle~ Q^{\vee}, 1^{k}\prod R_j ~\rrangle_{n\ell}^{X(\cc)} \Big).$$
(In case there is only one internal insertion, since there is no ancestor, we use the divisor equation to create a divisor insertion $t$, and then the following argument goes through.)

The cycle from $(\bp^1, \,\{0\})$ has dimension
$ c-1 - \sum (\nu_j -1 )$, which is non-zero only if  \[ c-1  \ge (\nu_j -1 ).\]
And if \[ c-1  = \sum (\nu_j -1 ), \]
the cycle is zero dimensional, and then the cycle from $X(\cc)$ is top dimensional.
From here we see $(c, ~\delta_{a}Q)$ dominates 
$\{ (\nu_j ,  ~\delta_{b_j}R_j)\}$.
\end{proof}

\subsection{Fiber class}

\begin{thm} \label{thm:fibertypei}
Fiber class type I invariants of $(P, P_0)$ and $(P,P_\infty)$ are $\T$-invariant.
\end{thm}

\begin{proof}
The proof is entirely similar to the proof of Theorem~\ref{thm:fibertypeii} and is omitted.
\end{proof}

\subsection{Positivity of certain relative two-pointed invariants on $(\bp^1, \{\infty \})$} \label{ss:nonzero}
Here we establish some positivity lemmas of certain relative integrals on $(\bp^1, \{\infty \})$.
These integrals naturally occur in the degeneration process as``coefficients'' of the highest order terms.
We need the non-vanishing in order for the inversion of degeneration to work.
(See, e.g., the proof of Proposition~\ref{prop:inversion'}.)

\begin{lemma} \label{=1}
For relative invariants of $(\bp^1, \{\infty \})$, we have
\[
\langle~ (c,1) \mid t^c \taub_{c-1}(t) ~\rangle_{c[\bp^1]} =1.
\]
\end{lemma}

\begin{proof}
Let $\mbar_{0,\,c+2}(\bp^1,\{ \infty\}; (c,1))$ be Kim's moduli stack of log maps of degree $c$ with 
one fully ramified marked point.

Consider the map
\[
\xymatrix{
\mbar_{0,c+2}(\bp^1,\{\infty\}; (c,1))\ar[rr]^{\prod ev_i \times \st} && \prod_{i=1}^{c+1} \bp^1 \times \mbar_{0, \,c+2},
}
\]
where $ev_i$'s are the $c+1$ evaluation maps determined by the internal marked points, and $\st$ is the the map stabilizing the source curve.
We have
\begin{equation*}
\begin{split}
&(\prod ev_i \times \st)_*[\mbar_{0,c+2}(\bp^1,\{ \infty \}; (c,1))]^{vir} \\
&\qquad = 
\langle~  (c,1)\mid t^c \bar{\tau}_{c-1}(t)~\rangle \Big[ ~\prod_{i=1}^{c+1} \bp^1 \times \mbar_{0, \,c+2}~\Big].
\end{split}
\end{equation*}
As $\mbar_{0,c+2}(\bp^1,\{ \infty \}; (c,1))$ is smooth of dimension $2c$, 
 \[ \langle~  (c,1)\mid t^c \bar{\tau}_{c-1}(t)~\rangle=
 \# (\prod ev_i\times \st)^{-1} \big\{ p_1, \cdots, p_{c+1}, (\bp^1, x_1, \cdots, x_{c+2})\big\} \]
 for 
generic $\big\{ p_1, \cdots, p_{c+1}, (\bp^1, x_1, \cdots, x_{c+2}) \big\}.$
It is not hard to see this equals one. 
(Consider $f: \bp^1 \to \bp^1$ such that $f^{-1}(\infty) = c x_{c+2}$, $f(x_i)=p_i$.
Assume that $x_{c+2} = \infty$, then such a map looks like 
\[f: [x,y] \mapsto \Big[~\sum_{i=0}^c a_i x^iy^{c-i}, y^c ~\Big].\] 
The constraints
$f(x_i)=p_i$ determine the coefficients $\{ a_i \}$ uniquely.)
\end{proof}

\begin{lemma} \label{l:6.4.2}
\[
 \langle~ (c,1) \mid t^{c'} \taub_{c-1}(t) ~\rangle_{c[\bp^1]} > 0\quad \mbox{when} \quad c' > c.
\]
\end{lemma}

\begin{proof}

Note that $t$ is ample on $\bp^1$ and $\psi$ classes are ample on $\mbar_{0,n}$ (by the stability condition). The divisor equation then implies
\[ 
\langle~ (c,1) \mid t^{c'} \taub_{c-1}(t) ~\rangle_{c[\bp^1]} \ge \Big(\int_{c[\bp^1]} t \Big)  \cdot \langle~ (c,1) \mid t^{c'-1} \taub_{c-1}(t) ~\rangle_{c[\bp^1]}.
\]
The lemma follows by induction and Lemma~\ref{=1}.
\end{proof}

\subsection{Non-fiber class}

\subsubsection{Reduction to non-distinguished insertions}
Consider Type I invariants of $(X(P),X(P_0))$ with $l \ge 1$ distinguished insertions. 

\begin{lemma} \label{lem:rmdist}
As before, $\omega$ stands for a non-distinguished insertion. We have
\begin{multline*}
\Big\langle~\ \vec{\nu}\ \mid \omega \cdot  \prod_{i=1}^{l}{i_0}_*(\alpha_i ) 
~\Big\rangle_{(\beta_P, \, d)}^{(X(P),X(P_0))}
=\\ \sum_{I}\sum_{\eta}
C_\eta \,\Big\langle~\ \vec{\nu}\ \mid \omega_1 \cdot  \prod_{i=1}^{l}{i_0}_*(\alpha_i )
\mid \mu, e^I \ ~\Big\rangle_{\Gamma_1}
\cdot
{p_X}_*( \langle~ \mu ,e_I \mid \omega_2    ~\rangle_{\Gamma_2} ).
\end{multline*}
\end{lemma}

\begin{proof}
This follows from the degeneration formula applied to the $\A^1$ family 
\[
 \left( W\big( X(P), X(P_\infty) \big), X(P_0) \times \A^1 \right)
\]
of pairs, with special fiber $(X(P), X(P_0)) \cup_{X(P_\infty)} X(P)$.
\end{proof}

\begin{cor} \label{c:6.5.2}
The $\T$-invariance of generating functions of type I without distinguished insertions and of type II
implies the $\T$-invariance of type I in general.
\end{cor}

\begin{proof}
Using (\ref{move}), we can assume all of the distinguished insertions are of the form ${i_0}_*(\alpha)$.
By Lemma~{lem:rmdist}, those generating functions are determined by type II invariants and type I invariants without distinguished insertions. 
\end{proof}

We will therefore assume that $\omega$ contains no disntinguished insertions in the remaining of this section.

\subsubsection{The case $\int_{\beta_P}  P_0 \ge 0$}
Consider a type I invariants of $(X(P), X(P_\infty))$:
\[
\langle~ \omega \mid \vec{\nu}~\rangle_{(\beta_P, \, d)}^{(X(P), X(P_\infty))},
\]
where 
\[
 \vec{\nu}= \{(\nu_j,\delta_{b_j}R_j )\},
\]
and $\delta_{b_j}R_j $ are taken from a $\T$-homogeneous basis of $H^*(X(D))$.

\begin{prop}
If $\int_{\beta_P}  P_0 \ge 0$ and $\omega$ has no distinguished insertions, then 
\[
\langle~ \omega \mid  \vec{\nu} ~\rangle_{(\beta_P, \, d)}^{(X(P), X(P_\infty))}=0.
\]
\end{prop}

\begin{proof}
This follows directly from  Proposition \ref{prop:svpf} (2).
\end{proof}

\subsubsection{The case $\int_{\beta_P}  P_0 < 0$}
For a weighted partition $\vec \nu$, 
$$k(\vec \nu) = \max\{0, \sum_{\{i \,\mid\, \nu_i >1\}}\nu_i - \on{Id}(\vec{\nu}) \}$$ 
is the number defined in (\ref{k-num}).
In the following, we apply divisorial insertion $[X(P_\infty)]$ to increase the number of internal marked point by ${k(\vec{\nu})}$,
in order to ensure the existence of the corresponding ancestors
$$\Big\langle \omega [X(P_\infty)]^{k({\nu})} \taub_{\nu-1}({i_\infty}_*( \cdot)) \Big\rangle_{(\beta_P, \,d)}^{X(P)} .$$

\begin{prop} \label{prop:inversion'}
Assume  $\int_{\beta_P}  P_0 < 0$.
\begin{enumerate}
\item 
If $\vec{\nu}$ is not empty, then there exists a positive number $C(\vec{\nu})$
such that
\[ 
C(\vec{\nu})\langle~ \omega \mid \vec{\nu} ~\rangle_{(\beta_P, \,d)}^{(X(P), X(P_\infty))} - \Big\langle~ \omega \cdot [X(P_\infty)]^{k(\vec{\nu})} \cdot \prod_{j}\taub_{\nu_j-1}({i_\infty}_*(\delta_{b_j}R_j)) 
~\Big\rangle_{(\beta_P, \,d)}^{X(P)} 
\]
is generated by generating functions of relative and rubber invariants on $X(P)$ of class at most $(\beta_P,d)$, and those of of $(X(P),X(P_\infty))$ involving class $(\beta_P,d)$ 
whose orders are lower than $\langle~ \omega \mid  \vec{\nu} ~\rangle_{(\beta_P, \,d)}$. 

\item 
If $\vec{\nu}$ is empty, then
\[
\langle~ \omega \mid  \vec{\nu} ~\rangle_{(\beta_P, \,d)}^{(X(P), X(P_\infty))} -
\langle~  \omega ~\rangle_{(\beta_P, \,d)}^{X(P)} 
\]
is generated by generating functions of relative invariants on $X(P)$ with curve classes lower than $(\beta_P,d)$.
\end{enumerate}
Here we say a formal power series $f$ is generated by $\{ f_m \}$ if it belongs to the subalgebra of formal power series generated by $\{f_m \}$.
\end{prop}

\begin{proof}
We prove the first part, and the second part can be proved similarly.

Consider  the family $W(X(P), X(P_\infty)) \to \mathbb{A}^1$
and a generating function of a general fiber
$$
\Big\langle~ 
\omega\cdot  [X(P_\infty)]^{k}\cdot\prod_{j} \taub_{\nu_{j-1}}({i_\infty}_*(\delta_{b_j}R_j)) 
~\Big\rangle_{(\beta_P, \,d)}^{X(P)}.
$$
We can lift the insertions $[X(P_\infty)]$ and ${i_\infty}_*(\delta_{b_j}R_j)$ to $(X(P), X(P_0))$ in the singular fiber. The degeneration formula allows us to express this function in terms of the ancestor relative invariants of $(X(P), X(P_\infty))$ and $(X(P), X(P_0))$.

By choosing the splitting properly, the resulting expression is as follows:
\begin{multline*}
\sum_{(\gamma_1,\gamma_2)} \sum_{\eta} 
C_\eta \,\langle~ \omega_1 \mid \mu ,  \gamma_1 ~\rangle ^{\bullet (X(P), X(P_\infty))}_{\Gamma_1} \times\\
{\pi_X}_*\left(\Big\langle~ \mu,  \gamma_2\mid\omega_2 \cdot [X(P_\infty)]^{k}\cdot 
\prod_j  \taub_{\nu_j - 1}([X(P_\infty)]\cdot \delta_{b_j}R_j))
 \Big\rangle^{\bullet (X(P), X(P_0))}_{\Gamma_2}\right).
\end{multline*}
Here $\pi_X: X(P) \to X(P)$ is induced from 
$$
\pi: P \to D \simeq P_\infty \hookrightarrow P,
$$ 
and $\gamma_1, \gamma_2$ are used to denote insertions without specifying their forms due to the behavior of ancestors.

Denote the curve classes of $\Gamma_i$ by $(\beta_i, d_i)$, then $\beta_2  < \beta_P$ since $\int_{\beta_P}  P_0 < 0$. We know that  $\beta_1 \le \beta_P$, and when $\beta_1 =\beta_P$, $\beta_2$ is a fiber class for $P \to D$.

When $\Gamma_1$ is connected, $\Gamma_2$ is a disjoint union of $l(\mu)$ rational curves. Denote by $p_j$ the marked point corresponding to the insertion ${i_{\infty}}_*(\delta_{b_j}R_j)$, and by $C_i$ the curve with relative condition $\mu_i$. If $\Gamma_2$ specifies that some $C_i$ has only two marked points and one of them is $p_j$ with $\nu_j >1$, then the ancestor $\taub_{\nu_j-1}$ appears in the relative insertion of $(X(P), X(P_\infty))$, and in this case we say $\Gamma_2$ is unstable (otherwise it is stable).

We divide $\eta = (\Gamma_1, \Gamma_2)$ into three types:
\begin{enumerate}
\item  $(\beta_1, d_1) <  (\beta_P,d)$, or $(\beta_1, d_1) = (\beta_P,d)$ and $\Gamma_1$ is not connected.
\item   $(\beta_1, d_1) = (\beta_P,d)$, $\Gamma_1$ is connected, and $\Gamma_2$ is stable.
\item $(\beta_1, d_1) = (\beta_P,d)$, $\Gamma_1$ is connected, and $\Gamma_2$ is unstable.
\end{enumerate}

It is easy to see that type (1) terms are generated by (connected) ancestor relative invariants
of $(X(P), X(P_0))$ and $(X(P), X(P_\infty))$ with curve classes less than $(\beta_P,d)$.
If we apply TRR to remove ancestor insertions then we get
primitive relative rubber invariants of $X(P)$ with curve classes less than $(\beta_P,d)$.

For type (2) and type (3) terms, the factor with discrete data $\Gamma_2$ is generated by relative and rubber invariants of 
$X(P)$ with class less than $(\beta_P, d)$.

We will show that for  type (2) and (3) terms, the $\Gamma_1$ factor might produce the term
\[ \langle~ \omega \mid  \vec{\nu} ~\rangle_{(\beta_P, \, d)}\] after removing ancestors using TRR.
And for its `coefficients', a priori  functions generated by lower order invariants,  is always a positive number.

A type (2) term can be written as 
\begin{equation*}
\begin{split}
&\langle~ \omega_1 \mid \mu ,  e^I ~\rangle ^{(X(P), X(P_\infty))}_{(\beta_P,\,d)} \times\\
&\quad
{\pi_X}_*\left(\Big\langle~ \mu,  e_I \mid\omega_2 \cdot [X(P_\infty)]^{k}\cdot 
\prod_j  \taub_{\nu_j-1}([X(P_\infty)]\cdot \delta_{b_j}R_j))
 \Big\rangle^{\bullet (X(P), X(P_0))}_{\Gamma_2}\right).
\end{split}
\end{equation*}

If $e^I = \{\delta_{a_i}Q_i\}_{1 \le i \le l(\mu)}$, its dual $e_I $ is given by $\{\delta_{a_i}^{\vee}Q_i^{\vee}\}$. Recall we abuse the notation to use $Q_i^{\vee}$ to denote the class on $[X(\cc)/G]$ which is dual to $Q_i$ in $X(\cc)$.

We will show that the order of 
\[ 
\langle~ \omega_1 \mid \mu ,  e^I ~\rangle ^{(X(P), X(P_\infty))}_{(\beta_P,\,d)} 
\]
 is no greater than 
\[ 
\langle~ \omega \mid\vec{\nu} ~\rangle ^{(X(P), X(P_\infty))}_{(\beta_P, \,d)}.
\]
We can assume $\omega_1 =\omega$ or $\omega_2$ is empty, for otherwise 
$\langle~ \omega_1 \mid \mu ,  e^I ~\rangle ^{(X(P), X(P_\infty))}_{(\beta_P,\,d)}$
is of lower order.

The factor
$$
\Big\langle~ \mu, e_I \mid  [X(P_\infty)]^{k}\cdot\prod_{j} \taub_{\nu_{j-1}}([X(P_\infty)]\cdot \delta_{b_j}R_j)) 
\Big\rangle^{\bullet (X(P), X(P_0))}_{\Gamma_2}
$$
is determined by fiber integrals  for the the following bundle:
$$
\xymatrix{
(X(P),X(P_0)) \ar[r]\ar[d] & ([\bp^1 /G_m], [\{0\}/G_m]) \times [X(\cc)/G] \ar[d]\\
D \ar[r] & BG_m \times BG,
} 
$$ 
where $G=GL_{r+1} \times GL_{r+1}.$

The contribution from $C_i$ is the generating function
$$
\Big\langle~ \mu_i , \delta_{a_i}^{\vee}Q_i^{\vee} \mid
[ X(P_\infty)]^{k_i}\cdot\prod_{p_j \in C_i} \taub_{\nu_{j-1}}([X(P_\infty)]\cdot \delta_{b_j}R_j)) 
 ~\Big\rangle^{(X(P), X(P_0))},
$$ 
where $k_i$ of those $k$ marked points with $X(P_\infty)$-insertion are distributed to $C_i$. If this fiber integral is nonzero, then  \[
\deg \delta_{a_i}^{\vee} + \sum_{p_j \in C_i}  \deg \delta_{b_j}
\le \dim D.\]
By summing up all $i$ we get  
\[
\sum_i \deg {\delta_{a_i}} = \sum_i (\dim D - \deg \delta_{a_i}^{\vee}) \ge \sum_j \deg{\delta_{b_j}}.
\]
This is the same as 
\[\deg_D(\vec{\mu}) \ge \deg_D(\vec{\nu}).\]

A potential highest order term should satisfy 
\[ 
\deg_D(\vec{\mu}) = \deg_D(\vec{\nu}),\] 
and then
\[
\Big\langle~ 
\mu_i , \delta_{a_i}^{\vee}Q_i^{\vee}  \mid
 [X(P_\infty)]^{k_i}\cdot\prod_{p_j \in C_i} \taub_{\nu_{j-1}}([X(P_\infty)]\cdot \delta_{b_j}R_j)
  ~\Big\rangle^{(X(P), X(P_0))}
 \]
simplifies to  
$$
\Big( \int_{D} 
\delta_{a_i}^{\vee} \prod_{p_j \in C_i} \delta_{b_j} \Big) \Big\langle (\mu_i, Q_i^{\vee}) \mid t^{k_i} \prod_{p_j \in C_i} \taub_{\nu_{j-1}}(tR_j) \Big\rangle^{(\bp^1, \, \{ 0\})\times X(\cc)}.
$$
By Lemma~\ref{lem:stable}, we see that 
$$
(\mu_i, ~\delta_{a_i}Q_i) \quad \mbox{dominates} \quad
\{ (\nu_j ,  ~\delta_{b_j}R_j)\}_{\{p_j \in C_i\}}.
$$
This implies that $\vec{\mu}$ dominates $\vec{\nu}$. As
\[ \sum \mu_i = \sum \nu_j = \int_{\beta_S} D, \] 
by Lemma \ref{lem:order},
$\vec{\mu}$ should be identified with $\vec{\nu}$.

In fact, if 
\[ (\nu_j, ~ \deg\delta_{b_j},  ~\deg R_j)) \ne (1, 0, 0), \] 
then $p_j$ must be distributed on $C_i$ for which 
$ (u_i,  \delta_{a_i}Q_i) 
=(\nu_j, \delta_{b_j}R_j)$.
When 
\[(\nu_j, ~ \deg\delta_{b_j},  ~\deg R_j )= (1, 0, 0),\] $p_j$ can be distributed freely.
By further taking into account of the non vanishing of $\taub_{\nu_j-1}$, 
we see that if $k(\vec{\nu})>0$, there is a unique configuration up to $\aut(\{(\nu_j, b_j, R_j \})$,
whose `coefficient' is determined by 
$$\prod_{\nu_j>1}\langle~ (\nu_j,1) \mid t^{\nu_j} \taub_{\nu_{j-1}}(t)~\rangle.$$
If $k(\vec{\nu})=0$, we have
$ \on{Id}(\vec{\nu}) -  \sum_{\nu_j >1}\nu_j $ extra $t$ insertions which can be distributed freely among $C_i$'s.
These are positive numbers by the lemmas in \ref{ss:nonzero}.

Now we turn to type (3) terms.
The $\Gamma_1$ factor of a type (3) term is of the form
$\langle~ \omega_1 \mid \mu ,  \gamma_1 ~\rangle ^ {(X(P), X(P_\infty))}_{(\beta_P,\, d)}$, and the $\Gamma_2$ factor is of the form
\[
\Big\langle~ \mu,  \gamma_2\mid\omega_2 \cdot [X(P_\infty)]^{k}\cdot 
\prod_j  \taub_{*}([X(P_\infty)]\cdot \delta_{b_j}R_j))
 ~\Big\rangle^{\bullet (X(P), X(P_0))}_{\Gamma_2}.
\]

If $C_i$ has only two marked points with a non relative marked point $p_j$ with $\nu_j>1$,
then its contribution to the $\Gamma_2$ factor is
\[
\langle~ \mu_i , \delta_{a_i}^{\vee}Q_i^{\vee}  \mid [X(P_\infty)]\cdot \delta_{b_j} R_j) 
~\rangle^{(X(P), X(P_\infty))}.
 \]  
As a fiber integral, it is nonzero only if  
\[\deg\delta_{a_i} \ge \deg\delta_{b_j}.\]
Together with the estimate for fiber integrals appearing in type (2) terms above,
we see that  \[\deg_D \vec{\mu} \ge \deg_D \vec{\nu}.\]

When $\deg_D \vec{\mu} = \deg_D \vec{\nu}$,
we have $\deg\delta_{a_i} = \deg\delta_{b_j}$, and 
\[
\langle~ \mu_i , \delta_{a_i}^{\vee}Q_i^{\vee} \mid [X(P_\infty)]\cdot \delta_{b_j} R_j) 
~\rangle^{(X(P), X(P_\infty))}
 \]  simplifies to
 \[
 \Big (\int_D \delta_{a_i}^{\vee}\delta_{b_j} \Big )\langle~(\mu_i, Q_i^{\vee}) \mid tR_j~\rangle^{(\bp^1, \, \{ 0\})\times X(\cc)}.
\]
So $\delta_{a_i} = \delta_{b_j}$, and then by Lemma~\ref{lem:unstable} it is nonzero 
if and only if 
\[\mu_i=1, \quad Q_i=R_j.\]

Therefore the ancestor relative insertions
in  $\langle~ \omega \mid \mu , \gamma~\rangle ^{(X(P), X(P_\infty))}_{(\beta_P, \, d)}$
are of the form $(1, \taub_{\nu_j-1}(\delta_{b_j}R_j))$.
Applying Proposition \ref{prop:ancestor} and necessary dominance results for type (2) terms established above, we  conclude the highest order term is
$\langle~ \omega \mid \vec{\nu} ~\rangle ^{(X(P), X(P_\infty))}_{(\beta_P, \, d)}$. Its coefficients are again positive numbers.
\end{proof}

\subsubsection{Conclusion of type I}
Finally we are ready to prove the $\T$-invariance for type I relative generating functions on $P$, assuming the $\T$-invariance on $D$.

\begin{thm}  \label{thm:typei}
$\T$-invariance for $D$ implies $\T$-invariance for $(P, P_0)$ and $(P, P_\infty)$. 
\end{thm}

\begin{proof}
For fiber curve classes this is proved in Theorem \ref{thm:fibertypei}, so we consider the case with non-fiber curve classes.

Assuming $\T$-invariance for invariants without distinguished insertions, it is easy to 
prove invariance inductively on the number of distinguished insertions using Lemma \ref{lem:rmdist}.

For invariants without distinguished insertions,
Proposition \ref{prop:inversion'}  allows us to perform induction using the partial ordering defined in Section~\ref{s:1.5.2},
expressing a type I invariant in terms of an absolute invariant,  type II invariants, rubber invariants and type I invariants of lower order. As Invariance for absolute invariants, type II and rubber invariants are shown, inductively the theorem is proved.
\end{proof}

\section{$\T$-invariance between relative and absolute invariants} 

\subsection{Relative implies absolute} \label{s:7.1}

We use the notations introduced in Section~\ref{subsec:blowup}.

\begin{prop}  \label{prop:abs2rel}
$\T$-invariance for the pair $(\tilde{S}, E)$ and $(P,E)$ implies the 
$\T$-invariance for $S$ with non-extremal $(\beta_S,d)$.

\end{prop}

\begin{proof}
Consider deformation to the normal cone for $Z \hookrightarrow S$.
The degeneration formula shows that
\begin{multline*}
\langle~ \alpha ~\rangle_{(\beta_S, \,d)}^{X(S)} = 
\sum_{I } \sum_{\eta} C_\eta\,
{\phi_X}_{*}
\left (
\langle~ \alpha_1 \mid \mu , e^I~\rangle ^{\bullet (X(\tilde{S}), X(E))}_{\Gamma_1}
\right) \cdot \\
{p_X}_*
\left(
\langle~ \mu, e_I \mid\alpha_2 ~\rangle^{\bullet (X(P), X(E))}_{\Gamma_2}
\right)
\end{multline*}
and 
\begin{multline*}
\langle~ \T(\alpha) ~\rangle_{(\beta_S, \,d)}^{X'(S)} = 
\sum_{I } \sum_{\eta } C_\eta\,
{\phi_{X'}}_{*}
\left(
\langle~ \T(\alpha_1) \mid \mu , \T(e^I) ~\rangle ^{\bullet (X'(\tilde{S}), \,X'(E))}_{\Gamma_1}
\right) \cdot \\
{p_{X'}}_*
\left(
\langle~ \mu, \T(e_I) \mid \T(\alpha_2) ~\rangle^{\bullet (X'(P), \,X'(E))}_{\Gamma_2}
\right ).
\end{multline*}

As $\phi_X(\ell) = p_X(\ell) =\ell$ and $\phi_{X'}(\ell') = p_{X'}(\ell') =\ell'$, the analytic continuations for $(\tilde{S}, E)$ and $(P, E)$ are compatible with that for $S$.
\end{proof}

\subsection{Primaries imply ancestors} 
This is achieved by the topological recursion relations (TRR) for ancestors,
which express the $\psi$-classes on genus zero moduli spaces of curves
in terms of the boundary classes.
Therefore, we have the following

\begin{lemma} \label{lem:TRR}
$\T$-invariance for all absolute (resp.\ relative) primary invariants with 
curve classes less than or equal to $(\beta, d)$ implies invariance for 
absolute (resp.\ relative) ancestors with curve classes
less than or equal to $(\beta, d)$.
\end{lemma}

The explicit form of TRR is not important.
In the absolute case, TRR is well known.
We include some discussions on TRR for relative invariants in the appendix.

\subsection{Relative invariants associated to extremal rays}
In this subsection we show the
$\T$-invariance for $(S,D)$ with extremal curve classes $(\beta_S,d)=(0,0)$.





\begin{prop} 
The relative generating functions of $(S,D)$ with $(\beta_S,d)=(0,0)$
are $\T$-invariant.
\end{prop}

\begin{proof}
Consider deformation to the normal cone for $X(D) \to X(S)$.
By the degeneration formula we have 
\[
\langle~ \omega ~\rangle^{X(S)}_{(0,\,0)} = 
\langle~ \omega ~\rangle^{(X(S),\,X(D))}_{(0,\,0)} 
+ \langle~ \omega ~\rangle^{(X(P),\,X(D))}_{(0,\,0)}.
\]

By Proposition~\ref{prop:svpf}, 
$$\langle~ \omega ~\rangle^{(X(P),\,X(D))}_{(0, \,0)}=0.$$
Therefor $\T$-invariance for $(S,D)$ with extreme curve classes follows from that for $S$
which was proved in Part I \cite{LLW1}.
\end{proof}

\subsection{Absolute implies relative} \label{s:inv}
Recall $D$ is a smooth divisor in $S$ and 
$$P=\bp(N\oplus \oo)$$ is a $\bp^1$ bundle over $D$. It has two sections $P_0$ and $P_{\infty} = \bp(N) = D$. 
We assume $\T$-invariance for $\bp^1$ bundles proved in previous sections.


Let $\vec{\nu} =\{ (\nu_j ,  ~\delta_{b_j}R_j)\}$ be a weighted partition and $i_X:X(D) \to X(S)$ be the inclusion. Again $k(\vec \nu)$ is the number defined in (\ref{k-num}).

\begin{prop} \label{prop:inversion}
Assume $(\beta_S,d)$ is non-extremal.
\begin{enumerate}
\item  
If $\vec{\nu}$ is non empty, then there exists a positive constant $C(\vec{\nu})$
such that 
\[
C(\vec{\nu}) \langle~ \omega \mid \vec{\nu} ~ ~\rangle ^{(X(S), \,X(D))}_{(\beta_S, \,d)}-
\Big\langle~
\omega\cdot  [X(D)]^{k(\vec{\nu})}\cdot\prod_{j} \taub_{\nu_{j-1}}({i_{X}}_*(\delta_{b_j}R_j)) 
~\Big\rangle_{(\beta_S, \,d)}^{X(S)}
\]
is generated by generating functions on $(X(S), X(D))$ of lower order, and relative and rubber invariants on $X(P)$.

\item 
If $\vec{\nu}$ is empty, then
\[
\langle~ \omega  ~\rangle_{(\beta_S, \,d)}^{X(S)}-
 \langle~ \omega \mid\vec{\nu} ~ ~\rangle ^{(X(S), \,X(D))}_{(\beta_S, \,d)}
\] 
is generated by generating functions on $(X(S), X(D))$ of lower order, and those of relative invariants on $X(P)$.
\end{enumerate}
\end{prop}

\begin{proof}
For (1), consider the family $W(X(S), X(D)) \to \mathbb{A}^1$ and
a generating function of a general fiber
$$
\Big\langle~ 
\omega\cdot  [X(D)]^{k(\vec{\nu})}\cdot\prod_{j} \taub_{\nu_{j-1}}({i_{X}}_*(\delta_{b_j}R_j)) 
~\Big\rangle_{(\beta_S, \,d)}^{X(S)}.
$$
We can lift the insertions $X(D)$ and ${i_{X}}_*(\delta_{b_j}R_j)$ to $X(P)$ in the singular fiber.
The degeneration formula allows us to express this function 
 in terms of the ancestor relative invariants of $(X(S), X(D))$ and $(X(P), X(D))$.
For (2), apply the degeneration formula to 
$ \langle~  \omega  ~\rangle_{(\beta_S, \,d)}^{X(S)}$.

Analyzing the invariants involved in the degeneration formula by the arguments in the proof of Proposition \ref{prop:inversion'}, the proposition is proved.
\end{proof}


\begin{thm}\label{thm:rel2abs}
Let $(S,D)$ be a smooth pair.
If $\T$-invariance holds for $S$ and $D$, then it holds for $(S,D)$ 
with non-extremal curve classes:
\[
\T\Big(
\langle~ \omega \mid\vec{\nu}~ ~\rangle ^{(X(S), \,X(D))}_{(\beta_S, \,d)}\Big)
=
\langle~ 
\T(\omega) \mid\T(\vec{\nu})~~\rangle ^{(X'(S), \,X'(D))}_{(\beta_S, \,d)},
\] 
where we define
$$\T(\vec{\nu}) = \Big\{\big (\nu_j , \T(\delta_{b_j}R_j) \big)\Big\}.$$
\end{thm} 

\begin{proof}
This can be proved inductively using Proposition \ref{prop:inversion}.
Using the argument for Proposition \ref{prop:inversion},
we can show that
\[
\Big \langle~ 
\T(\omega) \cdot  [X'(D)]^{k(\vec{\nu})}\cdot\prod_{j} \taub_{\nu_{j-1}}\Big({i_{X'}}_*\big(\T(\delta_{b_j}R_j)\big)\Big) 
~\Big \rangle_{(\beta_S, \,d)}^{X'(S)}\]
has a highest order term
\[
\langle~ 
\T(\omega) \mid \T(\vec{\nu}) ~\rangle ^{(X'(S),\, X'(D))}_{(\beta_S,\, d)}.\]

The lower order terms are $\T$-invariant by induction. The theorem now follows from combining Theorems \ref{thm:rubber}, \ref{thm:fibertypei} and \ref{thm:typei}. 
\end{proof}

\section{Conclusion of the proof of Theorem~\ref{t:main}} \label{TheProof}

To prove $\T$-invariance for $(S,F,F')$, we reduce the general case to the case when $F$ and $F'$ admit complete flags, then by deformation invariance of $\T$-invariance, to the split case.

\subsection{Motivation of the refined induction}

Before we go to the actual proof, we would like to motivate the (refined) induction procedure by looking at some starting cases. 

Let $(S, F, F')$ be the triple defining the local model of an ordinary flop. If $\dim S = 0$ this is the simple flop case and the $\T$-invariance was proved in \cite{LLW}. Indeed, if $\dim S = 1$ then $F$ and $F'$ admit complete flags and the $\T$-invariance is reduced to the split case proved in \cite{LLW2}. 

If $\dim S = 2$, $F$ and $F'$ may not admit complete flags. We need to perform a sequence of blow-ups $S_{i + 1} \to S_i$ ($S_0 = S$) to achieve this property. Fortunately each blow-up has only points as its center and the normal bundles are all trivial. In particular the easier quantum Leray--Hirsch (Theorem~\ref{QLH-fiber}) and the ``$\T$-invariance for $\bp^1$ bundles'' established in the previous sections all apply and the $\T$-invariance is again reduced to $S_i$ for $i$ large where the pullbacks of $F$ and $F'$ admits complete flags and the proof is done. 

Essentially the same argument applies to the case $\dim S = 3$: Let $T \subset S$ be the blow-up center. During the applications of degeneration formulas and deformation to the normal cone, the essential objects to take care are the normal bundles 
$$
N = N_{T/S} \quad \mbox{and} \quad 
N_{T\times \{0\}/(S \times \A^1)} = N \oplus \oo.
$$ 
If $\dim T \le 1$, then $N$ is either trivial or deformable to split bundles. If $\dim T = 2$, then $N$ is already a line bundle. In all cases the easier quantum Leray--Hirsch and $\T$-invariance for $\bp^1$ bundles apply and the proof is done without the need of any refinement to the blow-ups. 

The situation changes when $\dim S = 4$ and $S_1 \to S$ is the blow-up along $T$ with $\dim T = 2$. In this case $N = N_{T/S}$ may not admit complete flags anymore. The space $\bp_T(N)$ is a 3-fold whose $\T$-invariance can be assumed by induction. However, $P = \bp_T(N \oplus \oo)$ is also of the same dimension as $S$, with $N \oplus \oo$ being non-deformable to split bundles. In particular we are unable to deduce $\T$-invariance for $P$ from that for $T$ (which is known by induction since $\dim T < \dim S$) via Theorem~\ref{QLH-fiber}. Thus, an additional sequence of blow-ups on the surface $T$ is indispensable in order for the proof to proceed. 

Indeed, for each step of blow-up $S_{i + 1} \to S_i$ along some $T_i \subset S_i$, the normal bundle $N_i = N_{T_i/S_i}$ has to be treated similarly. 
It is therefore more economic to use a refined induction presented in the following subsection.

\subsection{Proof of the main theorem by reduction to split bundles} \label{s:8.1}

Given a triple $(T,G, G')$ and a finite set of vector bundles $N_1, N_2, \cdots, N_k$ over $T$, we will use the notation 
$$(T, G, G'; \{N_i\})$$ 
to represent a triple $(S, F, F')$, where
$$
S = \bp_{T}(N_1) \times_T \cdots \times_T \bp_{T}(N_k) \to T
$$
is the fiber product and $F$, $F'$ are the pullbacks of $G$, $G'$ from $T$ to $S$. We introduce this notation to streamline the induction argument in Theorem \ref{thm:inv}.

\begin{lemma} \label{l:flag}
$(T, G, G'; \{N_i\})$ becomes a triple with bundles admitting complete 
flags after a sequence of blowing-ups.
\end{lemma}

\begin{proof}
This is an immediate consequence of Lemma~\ref{l:5.3.1}.
\end{proof}

Let $Y$ be a smooth subvariety of $T$ and 
$$
\tilde T = {\rm Bl}_Y T \to T
$$ 
the blowing-up of $T$ along $Y$. Denote by $\tilde G$, $\tilde G'$, $\tilde N_i$ the pullbacks of $G$, $G'$, $N_i$ to $\tilde T$ respectively.

\begin{lemma} \label{lem:blowup}
$\T$-invariance for 
\[
(\tilde T, \tilde G, \tilde G'; \{\tilde N_i\}), \qquad (Y, G|_{Y}, G'|_{Y};
\{ {N_i}|_{Y}\}\cup\{N_{Y/T}\}),
\]
and 
\[(Y, G|_{Y}, G'|_{Y}; \{ {N_i}|_{Y}\}\cup\{N_{Y/T}\oplus \oo\})
\]
implies $\T$-invariance for $(T, G, G'; \{N_i\})$ with non-extremal $(\beta_S, \,d)$.
\end{lemma}

\begin{proof} 
The lemma follows from Proposition~\ref{prop:abs2rel} and Theorem \ref{thm:rel2abs}:

We have a natural projection 
$$
S=\prod_{T} \bp_{T}(N_i) \to T.
$$ 
Let $Z$ be the fiber product $S \times_T Y$, which is $\prod_{Y} \bp({N_i}|_{Y})$. Note that the normal bundle of $Z$ in $S$ is the pullback from $Y$ if its normal bundle in $T$. Let
$$
\tilde S = {\rm Bl}_Z S \to S
$$ 
be the blow-up of $S$ along $Z$ with exceptional divisor $E$. 

Recall Proposition~\ref{prop:abs2rel} says that invariance for $S$ follows from invariance for $(\tilde{S}, E)$ and $(P,E)$. By Theorem \ref{thm:rel2abs}, invariance of $(\tilde{S}, E)$ (resp. $(P,E)$) follows from those for $\tilde{S}$ and  $E$ (resp.$P$ and $E$). So invariance of $\tilde{S}, E$ and $P$ implies invariance for $S$, this is exactly what the lemma claims.
\end{proof}

\begin{thm} \label{thm:inv}
$\T$-invariance holds for all $(T, G, G'; \{N_i\})$ provided it holds for any 
triple with split vector bundles.
\end{thm}

\begin{proof}

If $(\beta_S, d)=(0, 0)$, $\T$-invariance was proved in Part I \cite{LLW1}.

For non-extremal generating functions with curve classes $(\beta_S, d)$, we will prove $\T$-invariance \emph{for all} 
$$
(T, G, G'; \{N_i\})
$$
by induction on the dimension of $T$. We simply call this statement as ``invariance for $T$''.

Note that by Proposition \ref{prop:definv}, we can assume $\T$-invariance holds for any triple with bundles admitting complete flags. When $\dim T \le 1$, $G$, $G'$ and $N_i$'s all admit complete flags, so $\T$-invariance holds for $(T, G, G')$ and then it holds for $(S, F, F')$ (i.e.~$(T, G, G'; \{N_i\})$) by Theorem \ref{QLH-fiber}. 

When $\dim T \ge 2$,  for any $Y$ a smooth subvariety of $T$, we may assume that invariance is proved for $Y$. Then by Lemma \ref{lem:blowup}, invariance for $T$ 
follows from that for $\tilde T$. Then by Lemma \ref{l:flag}, after a finite number of
blowing-ups, we are left to proving invariance for a triple with bundles admitting complete 
flags, which is guaranteed by our hypothesis.
\end{proof}

\begin{proof}[Proof of Theorem~\ref{t:main}]
Since $\T$-invariance of split bundles was proved in Part II \cite{LLW2},
Theorem~\ref{thm:inv} implies Theorem~\ref{t:main}.
\end{proof}

\subsection{Comments on algebraic cobordism of bundles on varieties}
\label{s:8.2}

A ``dream proof'' of Theorem~\ref{t:main} 
would be to apply the idea of the algebraic cobordism in \cite{mLP, LP3}.
Theorem~1 in \cite{LP3} implies that, up to double point degeneration, 
any list of vector bundles is equivalent to a $\mathbb{Q}$ combination
of split vector bundles on products of projective spaces,
whose genus zero theory can be easily computed as toric varieties.
Furthermore, as the quantum cohomology of a toric variety is semisimple,
the higher genus theory can be deduced from genus zero theory
by Givental's quantization formalism.

The major obstacle of this approach is the lack of 
``inversion of degeneration'' for a general double point degeneration.
Section~\ref{s:7.1} says that the $\T$-invariance is preserved 
under a ``forward degeneration''.
In Section~\ref{s:inv} we established the inversion of degeneration
for deformation to normal cones, which we will call
"backward DNC" for convenience.
Now, let us define a ``directed'' (non-reflexive) double point degeneration relation 
which allows forward degenerations and backward DNC.

\begin{question}
Can results similar to \cite{mLP, LP3} (in particular Theorem~1 in \cite{LP3})
be obtained by ``directed double point degeneration'' above?
\end{question}

If so, the story will be much simpler. At this moment,
we have no idea whether this is possible.
Nonetheless, the forward degenerations and backward DNC can lead us 
a bit further than we have employed in this paper.
For example, if $\dim S=1$, i.e.\ a curve,
one sees immediately that only forward double
point degenerations, which includes deformations, 
are needed to reduce the proof of $\T$-invariance of
$(S, F, F')$ to  to $(\bp^1, F, F')$, 
with $F$ and $F'$ being of the form 
$\oplus \mathscr{O}(k_i)$ and $\oplus \mathscr{O}(k'_i)$.
If we assume results in Sections~2-5 can be generalized to higher genera,
the above strategy will also apply to any genus.
Therefore, we just need to establish the $\T$-covariance for the 
\emph{absolute} invariants on $(\bp^1, F, F')$ in all genera.
However, since the local models built upon $(\bp^1, F, F')$ are
toric, $\T$-invariance of higher genera follows from that of
genus zero by the strategy in \cite{ILLW}.

For general $S$, the reduction in genus zero
can also go further than we have used in Section~\ref{s:8.1}.
Tracing the arguments in Sections~3.1 and 2.5 in \cite{LP3}, 
one sees that the triple $(S, F, F')$ can be reduced
to the following two special types by deformations to the normal cones only:

\begin{enumerate}
\item[(i)] $F$, $F'$ are direct sum of globally generated (bpf) line bundles;
\item[(ii)] $S = (\mathbb{P}^1)^{\times n}$, and 
$F= F' = \mathscr{O}(l_1, \ldots, l_{r+1})^{\oplus (r+1)}$.
\end{enumerate}

\begin{appendix}
\section{TRR for relative ancestors}


For $\Gamma = (0, [n+m], \beta, \mu)$, let $\K_{\Gamma}(X,D)$ be 
Kim's moduli stack of relative stable maps 
from genus 0, $n+m$ marked curve to $(X,D)$,
with curve class $\beta$, relative profile $\mu$;
$n$ is the number of internal (or non-relative) marked points, and $m=l(\mu)$.
When $n+m \ge 3$,
consider 
$$\rho^{\Gamma}: \K_{\Gamma}(X,D) \to \mbar_{0,n+m},$$
which maps a relative stable map to the stabilization of the source curve.

For a partition of $[n+m]$ into a disjoint union $A \sqcup B$, denote by 
\[
D_{A|B}: \mbar_{0, A \cup \{\circ_A\}}  \times \mbar_{0, B \cup  \{\circ_B\}}
\to
\mbar_{0,n+m}\]  
the map gluing $\circ_A$ and $\circ_B$. 
By abusing notation, we will also 
use $D_{A|B}$ to represent the corresponding divisor of $\mbar_{n+m}$.
We will use $\alpha, \delta$ to represent cohomology classes in  
$H^*(X), H^*(D)$ respectively, adding subscripts to number different classes,  
and superscript $\vee$ to represent a dual class with respect to 
some chosen basis.
We use the notation for the virtual cycles
\[
 \llangle~
 \prod_{j=1}^n\alpha_j \mid \prod_{i=1}^m \delta_i
 ~\rrangle^{(X,D)}_{\Gamma}
 := [\K_{\Gamma}(X,D)]^{\vir} \cap 
\prod_{j=1}^n ev_j^*(\alpha_j) \prod_{i=1}^m ev_i^*(\delta_i).
\]
For cycles on a moduli stack to a rubber target, 
we will use the same notation except adding a superscript $\sim$.
We might add a superscript $\bullet$ to emphasis that the source curve is 
possibly disconnected.

For even classes $\{\alpha_j\}_{1\le j \le n}$, $\{ \delta_i \}_{1 \le i\le m}$.
We have
\begin{equation} \label{e:A.0.1} 
\begin{split}
&\rho_*^{\Gamma} 
\big( \llangle~
\prod_{j=1}^n \alpha_j \mid \prod_{i=1}^m \delta_i
~\rrangle^{(X,D)}_{\Gamma} \big)
 \cap D_{A|B}  
 \\
= &\sum_{(\Gamma_A,\Gamma_B)}  \sum_{\alpha} (D_{A|B})_* \Big(
 \rho^{\Gamma_A}_* \llangle  ~\prod_{j \in A} \alpha_j
  \cdot \alpha \mid \prod_{i \in A} \delta_i ~\rrangle^{(X,D)}_{\Gamma_A} 
 \times \\ 
 & \qquad \qquad \rho^{\Gamma_B}_*
 \llangle~
  \prod_{j \in B} \alpha_j
  \cdot \alpha^{\vee} \mid \prod_{i \in B} \delta_i ~\rrangle^{(X,D)}_{\Gamma_B}
\Big) \\
+ &\sum_{\eta \in \Omega} 
\frac{m(\eta)}{|M|! \cdot c(\eta)}
\sum_{\{\delta_* \mid * \in M\}}
\rho^{(\Xi,\Xi')}_* \Big(
\llangle ~\prod_{j\in N} \alpha_{j} \mid \prod_{* \in M} \delta_*
~\rrangle^{\bullet (X,D)}_{{\Xi}}
\times \\
&\qquad  \qquad \llangle~
\prod_{* \in M} \delta_*^{\vee} \mid \prod_{j\in N'} \alpha_{j} \mid
\prod_{i=1}^m \delta_{i}
~\rrangle^{\bullet P\sim }_{{\Xi'}}
\Big).
\end{split}
\end{equation}

Let 
$$\Gamma_A=(0, A\cup \{\circ_A\}, \beta_A, \mu_A) \quad \text{and} \quad
\Gamma_B=(0, B\cup\{ \circ_B\}, \beta_B, \mu_B),$$
then $\circ_A$ and $\circ_B$ denote internal marked points.
The summation $\sum_{(\Gamma_A, \Gamma_B)}$ are over those 
$(\Gamma_A,\Gamma_B)$'s such that 
$\beta_A + \beta_B =\beta$, and $\mu$ is the disjoint union of 
$\mu_A$ and $\mu_B$.
The summation $\sum_{\alpha}$ runs over a basis of $H^*(X)$.

$\eta=(\Xi, \Xi')$ is a splitting of $\Gamma$ into 
two modular graphs with necessary compatibility conditions. 
See \cite[Definition 4.8.1, 5.1.1]{AF}.
$\Omega$ is the set of all
possible splittings.
$M$ is a labeling of the set of roots for $\Xi$ and $\Xi'$. (i.e. the set of marked points mapped into the divisor $D$ or $P_\infty$).
$N$ (reps.~$N'$) is a labeling of the set of legs of the modular graph $\Xi$ (reps.~$\Xi'$) (i.e. the set of the remaining marked points).
$m(\eta)$ is the product of the contact orders of the roots in $M$. By the compatibility of 
$\eta$ and $\Gamma$, there is a root $\circ_{\Xi}$ in $\Xi$ 
corresponding to a root  $\circ_{\Xi'}$ in $\Xi'$  such that if we glue all the roots but $\circ_{\Xi}$ and $\circ_{\Xi'}$, we get a disconnect curve in $\mbar_{0, A \cup \{\circ_A\}}  \times \mbar_{0, B \cup  \{\circ_B\}}$.
$c(\eta)$ is the contact order for the root $\circ_{\Xi}$ or $\circ_{\Xi'}$.
$\rho^{(\Xi,\Xi')}$ is the map 
\begin{multline*}
\K_{\Xi}(X,D) \times \K_{\Xi'}(P,P_0,P_\infty)^{\sim} \to \\
\prod_{v \in V(\Xi)} \mathfrak{M}_{g(v), n(v)} \times 
   \prod_{v \in V(\Xi')} \mathfrak{M}_{g(v),n(v)} 
\to \mbar_{0, n+m}.
\end{multline*} 
The first arrow is mapped to the source curve;
the second arrow is the gluing of all the roots.

Using these it is straightforward to deduce a TRR using 
\[ \psi_i = \sum_{i\in A, ~{j,k} \in B} D_{A|B}  \in A^*(\mbar_{0,n+m}) \]
Here $i,j,k$ are three indexes in $[n+m]$.   
These are the ancestors we consider in the paper.

\end{appendix}

\end{document}